\documentclass[a4paper,11pt]{article}

\usepackage[T1]{fontenc}

\usepackage{amssymb,latexsym,amsfonts}
\usepackage{amsmath}
\usepackage{graphicx}
\usepackage{geometry}
\usepackage[all]{xy}
  \newtheorem{theorem}{Th\'{e}or\`{e}me}
\newtheorem{proposition}{Proposition}
\newtheorem{lemma}{Lemme}
\newtheorem{definition}{D\'{e}finition}

\newtheorem{proof}{Preuve}
\usepackage[francais]{babel}
\usepackage[latin1]{inputenc}
\begin{document}
\title{Réalisation de de Rham des motifs de Voevodsky}
\author{Alexis Bouthier}
\date{12 prairial an 217 (01/06/09)}

\maketitle
$\\$$\\$$\\$$\\$$\\$$\\$$\\$$\\$$\\$$\\$$\\$$\\$$\\$$\\$$\\$$\\$$\\$$\\$$\\$$\\$$\\$$\\$$\\$$\\$$\\$$\\$$\\$$\\$$\\$$\\$$\\$$\\$$\\$$\\$$\\$$\\$$\\$$\\$$\\$$\\$
$\\$$\\$$\\$$\\$$\\$$\\$$\\$$\\$$\\$$\\$$\\$$\\$$\\$$\\$$\\$$\\$$\\$$\\$$\\$$\\$$\\$$\\$$\\$$\\$$\\$
\tableofcontents{}
\section{Introduction}

Il y a 40 ans, Grothendieck introduisait la th\'{e}orie des motifs purs, avec pour ambition  d'une part de r\'{e}aliser une th\'{e}orie cohomologique universelle qui d\'{e}crit la cohomologie des vari\'{e}t\'{e}s projectives lisses, dont les diff\'{e}rentes incarnations seraient les cohomologies de Weil qui sont \`{a} notre disposition: Betti, l-adique, rigide, de Rham, appel\'{e}s foncteur de r\'{e}alisation. D'autre part, il s'agissait d'\'{e}laborer une th\'{e}orie de Galois pour les nombres transcendants.

Cette th\'{e}orie d\'{e}bute sans doute avec le formalisme tannakien d\'{e}velopp\'{e} dans Saavedra [24], suivant les id\'{e}es de Grothendieck, o\`{u} \`{a} la fin est \'{e}nonc\'{e} tout un faisceau de conjectures sur les motifs dits purs. Cette th\'{e}orie est un analogue lin\'{e}aire de la th\'{e}orie profinie du $\pi_{1}$ d\'{e}velopp\'{e}e dans SGA 1. Un des r\'{e}sultats principaux est qu'une cat\'{e}gorie tensorielle k-lin\'{e}aire, muni d'un foncteur fibre neutre ([1,10,24] pour les d\'{e}finitions), dite tannakienne, est \'{e}quivalente \`{a} la cat\'{e}gorie de repr\'{e}sentations d'un groupe proalg\'{e}brique. La cat\'{e}gorie des motifs purs devrait former une telle cat\'{e}gorie et le groupe tannakien qui lui est canoniquement associ\'{e} serait le groupe de Galois motivique. Malheureusement, une telle cat\'{e}gorie n'est pas disponible pour le moment, principalement \`{a} cause des conjectures standards qui restent tr\`{e}s largement ouvertes depuis plus de 40 ans.

Dans les ann\'{e}es 80, Deligne et Beilinson \'{e}laborent un vaste programme pour construire la cat\'{e}gorie des  motifs mixtes qui eux, sont cens\'{e}s d\'{e}crire la cohomologie de toutes les vari\'{e}t\'{e}s. Deligne  propose de construire  une cat\'{e}gorie $DM_{gm}(k)$, la cat\'{e}gorie d\'{e}riv\'{e}e des motifs mixtes,  munie d'une t-structure naturelle o\`{u} l'on retrouverait les motifs mixtes en prenant le coeur de cette t-structure. Cette cat\'{e}gorie serait également munie de foncteurs de r\'{e}alisation, issus des diff\'{e}rentes th\'{e}ories cohomologiques. Construire la cat\'{e}gorie d\'{e}riv\'{e}e semble plus facile dans la mesure o\`{u} l'on regarde simultan\'{e}ment tous les groupes de cohomologie et donc on peut se dispenser des probl\`{e}mes li\'{e}s \`{a} l'alg\'{e}bricit\'{e} des projecteurs de Künneth. Une telle entreprise a \'{e}t\'{e} men\'{e}e \`{a} bien par  Voevodsky ([32]), Levine ([20]) et Hanamura qui ont propos\'{e} chacun une construction d'un cat\'{e}gorie avec des bonnes prori\'{e}t\'{e}s.

Il fallait maintenant munir chacune de ces constructions de r\'{e}alisations. Levine les a construit pour sa cat\'{e}gorie de motifs mixtes et donc, par comparaison avec la caract\'{e}ristique nulle avec la construction de Voevodsky, on dispose indirectement de tels foncteurs. N\'{e}anmoins, il apparaît pr\'{e}f\'{e}rable de ne pas passer par ce d\'{e}tour et d'obtenir les transferts directement sur les complexes concernés.

Huber, dans [14], construit un foncteur de r\'{e}alisation, en passant par une cat\'{e}gorie interm\'{e}diaire appel\'{e}e cat\'{e}gorie des syt\`{e}mes de r\'{e}alisations, valable pour tout corps k, qui nous fournit les r\'{e}alisations de Betti, l-adique et de Hodge (le cas rigide n'est pas trait\'{e}). N\'{e}anmoins, Deligne-Goncharov [10] signalent qu'en ce qui concerne la r\'{e}alisation de Hodge, la litt\'{e}rature n'est pas satisfaisante \`{a} ce sujet et \'{e}bauchent une construction pour l'obtention d'un tel foncteur.

Dans ce m\'{e}moire, il s'agit de construire un foncteur de r\'{e}alisation de de Rham pour les motifs mixtes g\'{e}om\'{e}triques de Voevodsky.  La m\'{e}thode pr\'{e}sent\'{e}e ici est celle esquiss\'{e}e dans Deligne-Goncharov. Elle  consiste \`{a} mettre des transferts sur le complexe de de Rham logarithmique  et une fois verifi\'{e}e l'invariance homotopique et la suite de Mayer-Vietoris pour la cohomologie de de Rham ainsi que quelques autres hypoth\`{e}ses mineures, on pourra conclure. Ce travail se d\'{e}compose de la façon suivante, premi\`{e}rement les propri\'{e}t\'{e}s et les conjectures qui existent au niveau des motifs purs pour motiver la construction; deuxi\`{e}mement, l'introduction de la cat\'{e}gorie des correspondances finies et des motifs mixtes g\'{e}om\'{e}triques. Ensuite, nous aurons besoin de r\'{e}sultats techniques sur la localisation des correspondances finies afin de mettre des transferts sur la r\'{e}solution de Godement d'un faisceau. Une fois ceci fait, il ne reste plus qu'\`{a} mettre des transferts sur la cohomologie de de Rham, cela s'\'{e}tendra alors aux r\'{e}solutions flasques canoniques et le reste ensuite ne sera plus que formel.

Il y a \`{a} noter que derni\`{e}rement, les travaux de Lecomte-Wach [19 bis] et de Cisinski-D\'{e}glise [5]  proposent une d\'{e}finition alternative d'un tel foncteur, mais dans les deux cas, pour le moment ils ne disposent pas des poids pour le foncteur de r\'{e}alisation.
 N\'{e}anmoins, en ce qui concerne Cisinski-D\'{e}glise, leur formalisme de cohomologie de Weil mixte leur permet en revanche d'obtenir une r\'{e}alisation rigide, mais \`{a} nouveau se pose le probl\`{e}me de la  bonne cat\'{e}gorie de coefficients (en particulier, le foncteur n'arrive pas dans la cat\'{e}gorie d\'{e}riv\'{e}e des F-isocristaux).

Je voudrais remercier M.Levine, qui a dirig\'{e} ce m\'{e}moire et P.Gille qui a bien accept\'{e} d'être mon directeur en France. Je remercie \'{e}galement, F.Lecomte, J.Riou et C.Soul\'{e} pour avoir accepté de participer à mon jury. Ce travail a \'{e}t\'{e} r\'{e}alis\'{e} durant mon s\'{e}jour \`{a} l'Universit\'{e} de Northeastern, Boston,  \`{a} laquelle je tiens \`{a} exprimer ma sinc\`{e}re gratitude. Enfin, merci \`{a} Franck Gabriel, mon compagnon d'infortune.

\section{Motifs purs}
Dans cette section, on rappelle un certain nombre de propri\'{e}t\'{e}s des motifs purs qui justifient en partie l'\'{e}tude des r\'{e}alisations pour les motifs mixtes. L'autre partie, consistant en l'\'{e}tude des r\'{e}gulateurs, sera  abord\'{e}e plus tard.
\subsection{Relations d'\'{e}quivalence}
On note $\mathcal{P}(k)$ la cat\'{e}gorie des sch\'{e}mas projectifs lisses sur k, un corps.
Dans le cas d'une intersection transverse, on sait d\'{e}finir le produit d'intersection de deux cycles alg\'{e}briques. La notion d'\'{e}quivalence ad\'{e}quate nous permet d'\'{e}tendre \`{a} tous les cycles ce produit d'intersection.
Pour tout anneau commutatif R, on pose $\mathcal{Z}^{r}(X)_{R}:=\mathcal{Z}^{r}(X)\otimes R$.

\begin{definition}
Une relation d'\'{e}quivalence $\approx$ sur les cycles alg\'{e}briques est dite ad\'{e}quate si elle v\'{e}rifie les axiomes suivants pour X,Y dans $\mathcal{P}(k)$:\\
(1) $\approx{}$ est compatible \`{a} la structure R-lin\'{e}aire et \`{a} la graduation sur les cycles.\\
(2) pour tous $\alpha$, $\beta \in \mathcal{Z}^{\bullet}(X) _{R}$, il existe $\alpha'\approx\alpha$ tel que $\alpha'$ et $\beta$ se coupent proprement.\\
(3) pour tout $\alpha$ dans $\mathcal{Z}^{\bullet}(X) _{R}$ et tout $\gamma \in \mathcal{Z}^{\bullet}(X\times Y) _{R}$ coupant proprement $pr_{1}^{*}(\alpha)$, on a:\\ $\alpha\approx 0\Rightarrow\gamma_{*}(\alpha):=pr_{2*}(\gamma.pr_{1}^{*}(\alpha))\approx 0$.
\end{definition}
On note $\mathcal{Z}^{\bullet}_{\approx}(X)_{R}:=\mathcal{Z}^{\bullet}(X)_{R}/\approx$ qui en fait une R-alg\`{e}bre gradu\'{e}e commutative.\\
Les \'{e}l\'{e}ments de $\mathcal{Z}^{\text{dim} X+r}_{\approx}(X\times Y)_{R}$ sont appel\'{e}s les correspondances de degr\'{e} r. La formule
\begin{center}
$g\circ f=p_{13*}(p_{12}^{*}(f).p_{23}^{*}(g))$
\end{center}
 fournit une loi de composition asociative (cf Fulton) pour les correspondances modulo $\approx$:\\
\begin{center}
$\mathcal{Z}^{\text{dim} X+r}_{\approx}(X\times Y)_{R}\otimes \mathcal{Z}^{\text{dim} Y+s}_{\approx}(Y\times Z)_{R}\rightarrow \mathcal{Z}^{\text{dim} X+r+s}_{\approx}(X\times Z)_{R}$
\end{center}

On a alors en particulier que $\mathcal{Z}^{\text{dim} X}_{\approx}(X\times X)_{R}$ est une alg\`{e}bre en g\'{e}n\'{e}ral non commutative. L'unit\'{e} est la classe de la diagonale et elle est munie d'une anti-involution donn\'{e}e par la transpos\'{e}e.\\

Exemples de relations d'\'{e}quivalence:\\
-\'{e}quivalence rationnelle: Un cycle $\alpha$ est dit rationnellement \`{a} 0 si il existe $\beta\in\mathcal{Z}^{\bullet}_{\approx}(X\times~\mathbb{P}^{1})_{R}$ tel que $\beta(0)$ et $\beta(\infty)$ et que $\alpha=\beta(0)-\beta(\infty)$. Ils donnent classiquement les groupes de Chow.\\
-\'{e}quivalence num\'{e}rique: Soit $\alpha\in\mathcal{Z}^{r}_{\approx}(X)_{R}$ on a: $\alpha\approx_{num} 0\Leftrightarrow \forall \beta\in\mathcal{Z}^{r}_{\approx}(X)_{R}$, $ deg(\alpha.\beta)=~0$.\\
-\'{e}quivalence homologique, (cf ci-dessous).
On va rappeler d'abord l'axiomatique des cohomologies de Weil.

\begin{definition} Soit F un corps de coefficients de caract\'{e}ristique nulle.
Une cohomologie de Weil est un foncteur $H:\mathcal{P}(k)^{op}\rightarrow VecGr_{F}$ avec les propri\'{e}t\'{e}s et donn\'{e}es additionnelles suivantes:\\

(1) H est un $\otimes$-foncteur, nous avons une formule de Künneth:\\
$H(X)\otimes H(Y)=H(X\times Y)$ pour X,Y dans $\mathcal{P}(k)$.$\\$$\\$
(2) H(X) est en degr\'{e}s positifs.$\\$$\\$
(3) $\forall X,Y \in \mathcal{P}(k)$, on a un isomorphisme canonique $H(X\coprod Y)= H(X)\oplus H(Y)$.
$\\$$\\$

(4) Le F-espace vectoriel $H^{2}(\mathbb{P}^{1})$ est de dimension 1, son dual est not\'{e} $F(1)$. Pour tout F-espace vectoriel V, n un entier, on pose $V(n)=V\otimes F^{\otimes n}$.$\\$$\\$

(5) $\forall X\in\mathcal{P}(k)$ de dimension d, il existe une application trace multiplicative $Tr:~H^{2d}(X)(d)\rightarrow~ F$ qui induit un accouplement parfait de Poincaré:
\begin{center}
$H^{2d-i}(X)(d)\times H^{i}(X)\rightarrow H^{2d}(X)(d)\rightarrow F$.
\end{center}
$\\$$\\$

(6) Il existe un morphisme classe de cycles $cl:CH^{\bullet}(X)\rightarrow H^{2*}(X)(*)$ contravariant en $X\in\mathcal{P}(k)$, compatible aux produits et normalis\'{e} par la trace de telle façon que l'application sur les z\'{e}ros-cycles est donn\'{e} par le degr\'{e}.$\\$$\\$
(7) $cl([\infty])$ est donn\'{e} par $H^{2}(\mathbb{P}^{1})(1)$.

$\\$
\end{definition}
\begin{definition}
Un cycle $x\in CH^{d}(X)_{F}$ est homologiquement \'{e}quivalent \`{a} z\'{e}ro (selon une cohomologie de Weil fix\'{e}e H) si $cl([x])=0$. Cela d\'{e}finit une \'{e}quivalence ad\'{e}quate.

\end{definition}

Exemples:
- Cohomologie de Weil classiques:
\begin{center}

(1) Pour tout premier l inversible dans k et $\bar{k}$ une cloture s\'{e}parable de k, la cohomologie \'{e}tale $H^{\bullet}_{et}(X_{\bar{k}},\mathbb{Q}_{l})$ est une cohomologie de Weil \`{a} coefficients dans $\mathbb{Q}_{l}$.
\end{center}
$\\$
\begin{center}
(2) Si k est de caract\'{e}ristique nulle, la cohomologie de de Rham alg\'{e}brique $H^{\bullet}_{dR}(X/k)$ est une cohomologie de Weil \`{a} coefficients dans k.
\end{center}
$\\$
\begin{center}

(3) Si $\sigma:k\rightarrow \mathbb{C}$ est un plongement complexe, la cohomologie de Betti $H^{\bullet}_{B}(X,\mathbb{Q})$ est une cohomologie de Weil \`{a} coefficients dans $\mathbb{Q}$.
\end{center}
$\\$
\begin{center}
(4) Si k est de caract\'{e}ristique p, la cohomologie cristalline $H^{\bullet}_{cris}(X)$ est une cohomologie de Weil à coefficients dans $W(k)[\frac{1}{p}]$.
\end{center}
$\\$
\begin{lemma}
$\approx_{rat}$ est l'\'{e}quivalence ad\'{e}quate la plus fine et $\approx_{num}$ la plus grossi\`{e}re.
\end{lemma}

\subsection{Cat\'{e}gories de motifs purs}

\begin{definition}

 On note $Mot_{\approx}^{eff}(k)$ l'enveloppe pseudo-ab\'{e}lienne de la cat\'{e}gorie suivante:
  Les objets sont les k-sch\'{e}mas lisses projectifs et les morphismes donn\'{e}s par $\mathcal{Z}^{dim X}_{\approx}(X\times~ Y)_{F}$ (si X connexe, sinon, faire la somme sur les composantes connexes) avec comme loi de composition, celle d\'{e}finie ci-dessus. Elle consiste en les motifs purs effectifs.
\end{definition}

\textbf{Remarque}:
\begin{center}
-On a un foncteur  naturel $h$ de $\mathcal{P}(k)$ de $Mot_{\approx}^{eff}(k)$ en prenant les mêmes objets et en associant \`{a} $f:X\rightarrow Y$ le graphe de f dont on a vu que c'\'{e}tait une correspondance de degr\'{e} z\'{e}ro.
\end{center}
\begin{center}
-Un point rationnel fournit une d\'{e}composition du motif de $\mathbb{P}^{1}$ : $h(\mathbb{P}^{1})= 1\oplus \mathbb{L}$ ou $\mathbb{L}$ est le motif de Lefschetz et 1 le motif de Spec(k), dit motif trivial.
\end{center}

\begin{definition}
En inversant le motif de Lefschetz, on obtient la cat\'{e}gorie des motifs purs. Elle est not\'{e}e $Mot_{\approx}(k)$. Dans le cas de l'\'{e}quivalence rationnelle, on parle de motifs de Chow, not\'{e} $CHM(k)_{\mathbb{Q}}$  et de motifs num\'{e}riques dans le cas de l'\'{e}quivalence num\'{e}rique.
\end{definition}

Etant donn\'{e} que l'\'{e}quivalence rationnelle est la plus fine, on a toujours un foncteur de:\\ $CHM(k)_{\mathbb{Q}}\rightarrow Mot_{\approx}(k)$.

Chacune des cohomologies de Weil  \`{a} coefficients dans K fournit un foncteur contravariant des motifs purs vers les K-espaces vectoriels gradu\'{e}s.
On a le diagramme suivant:\\
\begin{center}
\xymatrix{& CHM(k)_{\mathbb{Q}}\ar[ld]^{r_{et,l}}\ar[d]^{r_{DR}}\ar[rd]^{r_{B}}\ar[rrd]^{r_{cris}}\\VecGr_{\mathbb{Q}_{l}}&    VecGr_{k}&VecGr_{\mathbb{Q}}& VecGr_{W(k)[\frac{1}{p}]}}
\end{center}

On remarquera  que pour un corps de base fix\'{e} k, toutes ces cohomologies de Weil ne sont pas valables, par exemple celle de Betti en caract\'{e}ristique p.

Ces foncteurs se factorisent par la cat\'{e}gorie des motifs homologiques $Mot_{hom}(k)_{\mathbb{Q}}$.\\
Un motif pur M (pour une relation d'équivalence quelconque) est g\'{e}n\'{e}ralement repr\'{e}sent\'{e} par $ph(X)(r)$ o\`{u} p est un projecteur alg\'{e}brique et r un entier. \\
On a alors une structure tensorielle sur $Mot_{\approx}(k)$ donn\'{e}e par:\\
 $ph(X)(r)\otimes qh(Y)(s):=(p\otimes q)h(X\times X')(r+s)$.
 
 De plus, on a une notion naturelle de dual donn\'{e}e par:
 $ph(X)(r)\check{}=^{t}ph(X)(d-r)$ qui fait de $Mot_{\approx}(k)$ une $\otimes$-cat\'{e}gorie rigide.

Attardons nous un peu sur deux cat\'{e}gories de motifs particuli\`{e}res;
les  motifs homologiques et num\'{e}riques.
 Parmi tous les projecteurs, ceux qui joueront un rôle cl\'{e} seront ceux dit de Künneth; \'{e}tant fix\'{e} une cohomologie de Weil H, ils correspondent \`{a} la projection de $H^{\bullet}(X)$ vers $H^{i}(X)$, not\'{e}e $\pi_{i,X}$.

On a \`{a} ce propos une des conjectures standard suivantes:$\\$$\\$
\textbf{Conjecture} (Grothendieck):
Les projecteurs de Künneth $\pi_{i,X}$ sont  des correspondances alg\'{e}briques.

$\\$$\\$
Donnons quelques explications \`{a} propos de cette conjecture. Sous cette conjecture, les $\pi_{i,X}$ forment un système complet de projecteurs  d'orthogonaux. On pose $h^{i}(X)=\pi_{i,X}h(X)$. \\
On dispose alors de la graduation par le poids pour tout motif $M:=ph(X)(r)$ d\'{e}coupe sur X, on a une d\'{e}composition $M=\oplus ph^{2r+i}(X)(r)$.
Une telle d\'{e}composition vaut a fortiori pour une \'{e}quivalence plus grossi\`{e}re, en particulier pour l'\'{e}quivalence num\'{e}rique.

Dans le cas des motifs num\'{e}riques, on a la propri\'{e}t\'{e} suivante dûe a Jannsen:

\begin{theorem}(Jannsen, [16])
La cat\'{e}gorie $Mot_{num}(k)_{\mathbb{Q}}$ des motifs num\'{e}riques est ab\'{e}lienne semi-simple. On a de plus les \'{e}quivalences suivantes:\\
(i) $Mot_{\approx}(k)$ est ab\'{e}lienne semi-simple.\\
(ii) $\mathcal{Z}^{dim X}_{\approx}(X\times X)_{\mathbb{Q}}$ est une $\mathbb{Q}$-alg\`{e}bre semi-simple de dimension finie.\\
(iii) $\approx=\approx_{num}$.

\end{theorem}

Le lien entre les motifs homologiques et num\'{e}riques vient de la conjecture suivante.$\\$$\\$

\textbf{Conjecture} (Grothendieck): L'\'{e}quivalence homologique et l'\'{e}quivalence num\'{e}rique coïncident.\\

\textbf{Remarque:} On sait que cette conjecture est vraie pour les vari\'{e}t\'{e}s ab\'{e}liennes en caract\'{e}ristique nulle (cf [19 bis]).
La conjecture associ\'{e}e avec la proposition ci-dessus, ainsi que l'alg\'{e}bricit\'{e} des projecteurs, fait de la cat\'{e}gorie des motifs num\'{e}riques une cat\'{e}gorie tannakienne  semi-simple.

Comme on l'a vu, on a toute une famille de r\'{e}alisations des motifs purs num\'{e}riques vers des espaces vectoriels gradu\'{e}s, mais il s'av\`{e}re que l'on a des structures enrichies dans les cat\'{e}gories d'arriv\'{e}e que nous allons passer en revue.
$\\$$\\$
\textbf{Cohomologie \'{e}tale}:\\
On a une action naturelle du groupe de Galois $Gal(\bar{k}/k)$ sur $H^{\bullet}_{et}(X_{\bar{k}},\mathbb{Q}_{l})$  et pour tout entier q, on a donc une repr\'{e}sentation continue du groupe profini $Gal(\bar{k}/k)$,
 $H^{q}_{et}(X_{\bar{k}},\mathbb{Q}_{l})$.
De plus, $\mathbb{Q}_{l}(1)$ est isomorphe en tant que repr\'{e}sentation galoisienne a $\mathbb{Z}_{l}(1)[\frac{1}{l}]$ ou $\mathbb{Z}_{l}(1):=~\varprojlim_{\nu} \mu_{l^{\nu}}(\bar{k})$.
On a alors que le foncteur de r\'{e}alisation de $Mot_{rat}(k)\rightarrow VecGr_{\mathbb{Q}_{l}}$, s'enrichit dans la cat\'{e}gorie des  $\mathbb{Q}_{l}$-repr\'{e}sentations continues de $Gal(\bar{k}/k)$.
$\\$$\\$
\textbf{Cohomologie de de Rham alg\'{e}brique}:\\
On a une suite spectrale d'hypercohomologie
$E_{2}^{pq}=H^{q}(X,\Omega^{p}_{X/k})\Rightarrow H^{p+q}_{DR}(X/k)$
On a alors canoniquement une filtration F issue de la suite spectrale sur $H^{i}_{DR}(X/k)$.
Le calcul de $H^{2}(\mathbb{P}^{1}/k)$ montre qu'il est isomorphe \`{a} k en degr\'{e} de Hodge 1.$\\$$\\$

\textbf{R\'{e}alisation de Betti:}\\
k un sous-corps de $\mathbb{C}$, $X\in\mathcal{P}(k)$. Par la th\'{e}orie de Hodge, pour tout $i\geq 0$, on a une d\'{e}composition de $H^{i}_{DR}(X(\mathbb{C}))$. Par le th\'{e}or\`{e}me de comparaison Betti-de Rham, on a une d\'{e}composition canonique:
\begin{center}
$H^{i}_{B}(X)\otimes \mathbb{C}=\oplus_{p+q=i}H^{p,q}$ avec $\bar{H^{p,q}}=H^{q,p}$
\end{center}

\begin{definition}
Une $\mathbb{Q}$-structure de Hodge pure de poids n est un $\mathbb{Q}$-espace vectoriel de dimension finie avec une d\'{e}composition du complexifi\'{e}:
$V_{\mathbb{C}}=\oplus_{p+q=n} V^{p,q}$ avec $\bar{V^{p,q}}=V^{q,p}$. La filtration de Hodge sur $V_{\mathbb{C}}$ est d\'{e}finie par $F^{p}V_{\mathbb{C}}=\sum_{p'\geq p}{V^{n-p',p'}}$.

\end{definition}
On note $SH_{\mathbb{Q}}$ la cat\'{e}gorie des structures de Hodge pures. C'est une cat\'{e}gorie tannakienne avec comme foncteur fibre, le foncteur d'oubli.
Cette fois-ci le foncteur s'enrichit \`{a} travers $SH_{\mathbb{Q}}$.$\\$

\textbf{Cohomologie cristalline}:\\
Les groupes de cohomologie cristalline $H^{\bullet}_{cris}(X)=H^{\bullet}_{cris}(X/W(k))[\frac{1}{p}]$ sont canoniquement munis d'une action semi-lin\'{e}aire du Frobenius $\phi:H^{\bullet}_{cris}(X)\rightarrow H^{\bullet}_{cris}(X)$.

\subsection{Groupes de Galois motiviques}

Commencons par quelques rappels de th\'{e}orie tannakienne (cf [9],[24] pour plus de d\'{e}tails). Soit F un corps et soit $\mathcal{T}$ une $\otimes$-categorie rigide ab\'{e}lienne avec $End(1)=F$. Un foncteur fibre sur $\mathcal{T}$ est un $\otimes$-foncteur exact fid\`{e}le:\\
$\omega:\mathcal{T}\rightarrow Vec_{K}$ o\`{u} K est une extension de F.
Si un tel foncteur existe, on dit que $\mathcal{T}$ est une cat\'{e}gorie tannakienne.

On d\'{e}finit alors un K-sch\'{e}ma en groupes affine $G=Aut^{\otimes}\omega$.
Dans le cas où le foncteur fibre est neutre, i-e K=F, $\omega$ s'enrichit en une \'{e}quivalence de $\otimes$-cat\'{e}gories rigides:\\
$\omega:\mathcal{T}\rightarrow Rep_{F}G$
o\`{u} $Rep_{F}G$ d\'{e}signe la $\otimes$-cat\'{e}gorie rigide des F-repr\'{e}sentations de dimension finie de G. En outre, la cat\'{e}gorie des foncteurs fibres sur F est \'{e}quivalente \`{a} la cat\'{e}gorie des G-torseurs.

On a en plus un dictionnaire remarquable entre les propri\'{e}t\'{e}s algebriques de G et celles cat\'{e}goriques de $\mathcal{T}$:\\

(1)  F est de caract\'{e}ristique nulle, $\mathcal{T}$  semi-simple $\Leftrightarrow$ G  pro-r\'{e}ductif.
$\\$

(2) G est alg\'{e}brique $\Leftrightarrow$ $\mathcal{T}$ est de $\otimes$-g\'{e}n\'{e}ration finie.\\
On a \'{e}galement d'autres crit\`{e}res pour d\'{e}terminer la connexit\'{e} ou la finitude du groupe tannakien (cf [9], [24]).

On a \'{e}galement des propri\'{e}t\'{e}s analogues pour les morphismes.

Si $F:\mathcal{T}'\rightarrow \mathcal{T}$ est un $\otimes$-foncteur exact entre cat\'{e}gories tannakiennes neutres. On note $\omega$  un foncteur fibre pour $\mathcal{T}$, alors $\omega':=\omega\circ F$ est un foncteur fibre pour $\mathcal{T}'$ et on a un morphisme $f:G\rightarrow G'$.

R\'{e}ciproquement, si on a un morphisme  f entre les groupes tannakiens, on obtient un $\otimes$-foncteur exact $\phi:=f^{*}$.
On a:\\
\begin{center}
(i) f immersion ferm\'{e}e $\Leftrightarrow$ tout objet M de $\mathcal{T}$ est sous-quotient de l'image par $\phi$ d'un objet N' de $\mathcal{T}'$.
\end{center}
\begin{center}
(ii) f fid\`{e}lement plat $\Leftrightarrow$ $\phi$ est pleinement fid\`{e}le, et pour tout objet M' de $\mathcal{T}'$, tout sous-objet de $\phi(M')$ est l'image par $\phi$ d'un sous- objet M' de $\mathcal{T}'$.
\end{center}

La th\'{e}orie tannakienne apparaît \`{a} divers endroits en math\'{e}matiques, th\'{e}orie de Galois diff\'{e}rentielle, correspondance de Riemann-Hilbert, g\'{e}om\'{e}trie alg\'{e}brique et il ne s'agit pas ici de faire un panorama de la th\'{e}orie tannakienne pure. Nous allons juste nous attarder sur le  groupe de Galois motivique.
$\\$$\\$
On se place sous la conjecture $\approx_{num}=\approx_{hom}$ et on suppose l'algébricité des projecteurs. On consid\`{e}re la cat\'{e}gorie des motifs num\'{e}riques rationnels. On a vu qu'une telle cat\'{e}gorie est tannakienne neutre semi-simple avec comme foncteur fibre la r\'{e}alisation de Betti $H_{B}$.
On appelle alors groupe de Galois motivique $G_{mot, k}$ attaché à $H_{B}$  le sch\'{e}ma en groupes affines des automorphismes du foncteur fibre. C'est un  $\mathbb{Q}$-groupe pro-r\'{e}ductif.
$\\$$\\$
Etant donn\'{e} un motif num\'{e}rique M, son groupe de Galois motivique  $G_{mot, k}(M)$ est le groupe tannakien associ\'{e} \`{a} la cat\'{e}gorie tannakienne qu'il  engendre. C'est un groupe r\'{e}ductif sur $\mathbb{Q}$. De plus, on a que $G_{mot, k}=\varinjlim_{M} G_{mot, k}(M)$.
On a les propri\'{e}t\'{e}s suivantes pour les groupes de Galois motiviques:$\\$
(1) Pour un sous-corps k de $\mathbb{C}$, de clôture alg\'{e}brique $\bar{k}$, on a $G_{mot}(X_{\bar{k}})=G_{mot}(X_{\mathbb{C}})$ qui est d'indice fini dans $G_{mot}(X)$.$\\$
(2) On appelle groupe de Mumford-Tate $MT(X)$ le groupe tannakien associ\'{e} \`{a} la structure de Hodge pure $H_{B}(X)$. Ce sont des groupes connexes et la th\'{e}orie tannakienne fournit un morphisme de $G_{mot}(X)\rightarrow MT(X)$.
Conjecturalement, ce sont des groupes isomorphes. On a un cocaract\`{e}re de poids $\mathbb{G}_{m}\rightarrow MT(X)$, qui vient de la filtration par le poids.$\\$$\\$

Nous allons maintenant donner quelques exemples de calculs de groupes de Galois motiviques (cf [1] pour plus d'exemples).
$\\$
Par la th\'{e}orie des invariants de Chevalley, on a que si $\mathcal{T}$ est un cat\'{e}gorie tannakienne semi-simple neutre engendr\'{e}e par un objet M, munie d'un foncteur fibre $\omega$.
$Aut^{\otimes}\omega$ est le sous-groupe r\'{e}ductif ferm\'{e} de $GL(\omega(M))$.$\\$

Soit A une courbe elliptique non CM sur un corps k. Elle v\'{e}rifie les conjectures standards. On a $h(A)=\wedge h^{1}(A)$ o\`{u} 
$h^{1}(A)$ est de rang 2. On a donc que $G_{mot}(A)$ est un sous-groupe r\'{e}ductif de $GL_{2,\mathbb{Q}}$.
On va commencer par d\'{e}terminer $MT(A_{\mathbb{C}})$. C'est un  sous-groupe r\'{e}ductif connexe de $GL_{2}$ contenant les homoth\'{e}ties (l'image du cocaractere de poids). Il ne reste que trois possibilit\'{e}s: $\mathbb{G}_{m}$, $GL_{2}$ ou un tore de Cartan.
Ces groupes peuvent se distinguer par leur commutant. On a:
$End_{MT(A_{\mathbb{C}})}H_{B}(A)=End(A_{\mathbb{C}})\otimes\mathbb{Q}$.
On en d\'{e}duit alors que le groupe de Galois motivique est $GL_{2,\mathbb{Q}}$.$\\$

Les groupes de Galois motiviques ont beaucoup de propri\'{e}t\'{e}s conjecturales dont nous n'aurons pas l'occasion de parler, par exemple, la dimension du groupe de Galois motivique sur $\mathbb{Q}$ est \'{e}gale au degr\'{e} de transcendance de la $\mathbb{Q}$-alg\`{e}bre des p\'{e}riodes.
Dans le cas ci-dessus, cela nous donne que les p\'{e}riodes de la courbe elliptique sont transcendantes et ind\'{e}pendantes sur $\mathbb{Q}$.
Cette conjecture est appel\'{e}e conjecture des p\'{e}riodes de Grothendieck et englobe conjecturalement un grand nombre de r\'{e}sultats de la th\'{e}orie des nombres transcendants. On se référera au livre d'André [1] pour plus de précisions.
$\\$

\section{Motifs mixtes g\'{e}ometriques de Voevodsky}

\subsection{Correspondances}$\\$
Soit k un corps. On note Sm/k la cat\'{e}gorie des sch\'{e}mas lisses s\'{e}par\'{e}s sur k.
\begin{definition}Soit X un sch\'{e}ma connexe et lisse sur k, Y un k-sch\'{e}ma s\'{e}par\'{e}. Une correspondance \'{e}l\'{e}mentaire $\Gamma$ de X vers Y est un sous-sch\'{e}ma ferm\'{e} int\`{e}gre de $X\times Y$ qui est fini surjectif sur X.$\\$ Dans le cas X non connexe, une correspondance \'{e}l\'{e}mentaire est une correspondance d'une des composantes connexes.$\\$

Le groupe Cor(X,Y) est le groupe ab\'{e}lien libre engendr\'{e} par les correspondances \'{e}l\'{e}mentaires de X vers Y. On note  $Cor(X,Y)^{eff}$ les correspondances dites effectives, i-e avec tous les coefficients entiers positifs.
\end{definition}

\textbf{Exemples}: - Si $f: X\rightarrow Y$ alors la somme des composantes connexes de $\Gamma_{f}$ le graphe de f est une correspondance finie.$\\$
-Chaque sous-sch\'{e}ma ferm\'{e}  de $X\times Y$ fini, surjectif, d\'{e}finit une correspondance finie de X vers Y, il suffit de prendre la somme $\sum{ n_{k}[W_{k}]}$ o\`{u} les $W_{k}$ sont les composantes irr\'{e}ductibles de Z et  les $n_{k}$ les multiplicit\'{e}s g\'{e}om\'{e}triques de 
$W_{k}$ dans Z. $\\$

On va maintenant d\'{e}finir gr\^{a}ce au produit d'intersection une loi de composition pour les correspondances finies, valables \'{e}galement pour les cycles alg\'{e}briques.

Soit $\Gamma$ une correspondance de X vers Y et $\Gamma '$ une correspondance de Y vers Z, on d\'{e}finit $\Gamma \circ \Gamma'$ par:
$p_{13*}(p_{12}^{*}(\Gamma).p_{23}^{*}(\Gamma'))$ o\`{u} . d\'{e}signe le produit d'intersection de Serre.

On a que c'est une loi de composition associative et bilin\'{e}aire ([12]).
\begin{definition} On note SmCor/k la cat\'{e}gorie des correspondances finies o\`{u} les objets sont les sch\'{e}mas lisses s\'{e}par\'{e}s sur k et les morphismes les correspondances finies munies de la loi de composition donn\'{e}e ci -dessus.
\end{definition}

\begin{proposition} On a un foncteur fid\`{e}le $\Gamma: Sm/k \rightarrow SmCor/k$ donn\'{e} par:
 $X\rightarrow X$ et$\\$ $(f:X\rightarrow Y)\rightarrow \Gamma_{f}$ .
\end{proposition}
\begin{proof}
On a juste \`{a} v\'{e}rifier que $\Gamma_{f\circ g}=\Gamma_{f} \circ \Gamma_{g}$ ce qui est facile et laiss\'{e} en exercice (cf [12]).
\end{proof}

On a  de plus que SmCor/k est munie d'une structure tensorielle donn\'{e}e par $X\otimes Y:=~X\times Y$ et le produit tensoriel de morphismes est donn\'{e} par le produit externe de cycles, qui en fait une cat\'{e}gorie monoïdale sym\'{e}trique. On voit \'{e}galement que cette cat\'{e}gorie est naturellement additive avec $[X]\oplus[Y]=[X\coprod Y]$.\\

\textbf{Remarque}: Ces d\'{e}finitions se g\'{e}n\'{e}ralisent ais\'{e}ment \`{a} une base noeth\'{e}rienne reguli\`{e}re quelconque (cf [15]), avec quelques difficult\'{e}s suppl\'{e}mentaires pour la loi de composition si la base n'est pas r\'{e}guli\`{e}re.
On a \'{e}galement une deuxi\`{e}me d\'{e}finition des correspondances finies qui s'av\`{e}re parfois plus ais\'{e}e, qui convient surtout \`{a} la caract\'{e}ristique nulle.$\\$
Soit X dans Sm/k de caract\'{e}ristique nulle, on pose:$\\$ $Sym X=\coprod_{n\geq 1} Sym^{n}(X)$. $\\$On a alors que pour Y dans Sm/k, $(Sym X)(Y)= Hom (Sym(X),Sym(Y))$ où le Hom est au sens des monoïdes.

On obtient la proposition suivante:

\begin{proposition} Il y a un morphisme de monoïdes compatible \`{a} la composition des correspondances:$\\$ $\sigma: Cor(X,Y)^{eff}\rightarrow (Sym X)(Y)$ qui est un isomorphisme.
\end{proposition}

On rappelle la preuve de [4] Prop 2.1.2.
\begin{proof}

(i) On peut supposer X connexe.  On d\'{e}finit $\sigma$ sur les g\'{e}n\'{e}rateurs de Cor(X,Y). Soit $\Gamma$ une correspondance \'{e}l\'{e}mentaire de X vers Y. On note n le degr\'{e} de $\Gamma$ et on \'{e}crit $\Gamma= Spec(\mathcal{A}_{X})$ o\`{u} $\mathcal{A}_{X}$ est une $\mathcal{O}_{X}$-alg\`{e}bre coh\'{e}rente. $\\$Soit U un ouvert de X sur lequel $\Gamma$ est plat. Le X-sch\'{e}ma $Sym^{n}(\Gamma)$ admet une section $s_{\Gamma}$ au-dessus de U. Pour $f\in \mathcal{A}_{U}$, la fonction $s_{\Gamma}^{*}f^{\otimes n}$ est le d\'{e}terminant de f agissant par multiplication sur le $O_{U}$-module $\mathcal{A_{U}}$.$\\$ Comme X est normal, la section s'\'{e}tend de mani\`{e}re unique \`{a} X et on associe donc \`{a} $\Gamma$ le morphisme $\sigma(\Gamma): X\stackrel{s_{\Gamma}}{\rightarrow}Sym^{n}(\Gamma)\rightarrow Sym^{n}(Y)$.$\\$

(ii) Soit $\gamma_{V} \in Cor(V,Y)$ avec V ouvert non vide  de X. Montrons que $\gamma_{V}$ s'\'{e}tend \`{a} Cor(X,Y) si et seulement si  $\sigma(\gamma_{V}): V\stackrel{s_{\gamma_{V}}}{\rightarrow}Sym^{n}(\Gamma)\rightarrow Sym^{n}(Y)$, $n= deg(\gamma_{V})$ s'\'{e}tend \`{a} X.$\\$ Soit $Z_{V}$ le support de cycle $\gamma_{V}$, Z son adh\'{e}rence dans $X\times Y$. Alors $\gamma_{V}$ s'\'{e}tend \`{a} X si et seulement si Z est fini sur X. Supposons que $\sigma(\gamma_{V})$ s'\'{e}tend \`{a} $X \rightarrow Sym^{n}(Y)$. On consid\`{e}re $Y^{n}$ comme un sch\'{e}ma fini sur  $Sym^{n}(Y)$, et soit
$\tilde{Z}:=Y^{n}\times_{Sym^{n}(Y)}X$ et Z' l'image dans $X\times Y$ par n'importe quelle projection de $Y^{n}\rightarrow Y$. On a alors que Z' est fini sur X et $Z'\supset Z_{V}$ donc Z est fini sur X, ce qu'on voulait.$\\$

(iii) D'apr\`{e}s (ii), il nous suffit de v\'{e}rifier l'isomorphisme au point g\'{e}n\'{e}rique $\eta$ de X.$\\$ En remplaçant k par $\eta$ et Y par $ Y_{\eta}$, on peut supposer $X=Spec(k)$, donc Cor(X,Y) est le groupe des z\'{e}ros-cycles de Y. Comme c'est clairement un isomorphisme pour k alg\'{e}briquement clos, on a par la th\'{e}orie de Galois que c'est vrai si k est parfait, le cas g\'{e}n\'{e}ral s'en d\'{e}duit en prenant la cl\^{o}ture parfaite.$\\$ En effet, pour une extension finie k'/k de corps les compositions des morphismes canoniques de $Z_{0}(Y)\rightleftharpoons Z_{0}(Y_{k'})$ font la multiplication par [k':k]  et $(Sym X)(k)\rightarrow (Sym X)(k')$ est injective. $\\$

(iv) Soit $\gamma$ dans $Cor(X,Y)^{eff}$ et $\gamma'$ dans $Cor(Y,Z)^{eff}$ on a $\sigma(\gamma'\gamma)=\sigma(\gamma')\circ\sigma(\gamma)$ o\`{u}   $\sigma(\gamma)$ est vu comme un morphisme de Sym(X) vers Sym(Y). Comme en (3), on se ram\`{e}ne \`{a} X=Spec(k) et $\gamma=y \in Y(k)$. On peut supposer $\gamma'$ effective et en remplacant Y par un voisinage affine de $\gamma'(y)$, que  Z=Spec(B).
Alors  $z:=\sigma(\gamma'\gamma)$ et $z':=\sigma(\gamma')\circ\sigma(\gamma)$ sont des k-points de $Sym^{n}(Z)=Spec(B)^{\Sigma_{n}}$.$\\$ Cette alg\`{e}bre est engendr\'{e}e par les $b^{\otimes n}$ qui ne s'annule pas en $\gamma'(y)$ et il faut voir que $b^{\otimes n}(z)=b^{\otimes n}(z')$. 
Si $\gamma'(y)=\sum{n_{i}[z_{i}]}$, alors $b^{\otimes n}(\gamma(y))=\prod{Nm_{k(z_{i})/k}(b(z_{i})^{n_{i}})}$ o\`{u} $Nm_{k(z_{i})/k}:k(z_{i})\rightarrow k$ est l'application norme. De plus, $b^{\otimes n}(z')$ est le d\'{e}terminant de la multiplication par b sur la fibre d\'{e}riv\'{e}e $Li^{*}_{y}\mathcal{A}_{X}$  de $\mathcal{A}_{X}$ en y. C'est un complexe de B-modules  support\'{e} en les ${z_{i}}$ et de caract\'{e}ristique d'Euler-Poincar\'{e} en $z_{i}$ les $n_{i}$.
\end{proof}$\\$
\subsection{Faisceaux et transferts}
\begin{definition}
Un pr\'{e}faisceau avec transferts est un foncteur additif contravariant\\ $F:SmCor/k\rightarrow Ab$. On note PST(k) la cat\'{e}gorie de tels foncteurs. Pour ce paragraphe, on se ref\'{e}rera \`{a} [32].
\end{definition}
\textbf{Remarque}: D'apr\`{e}s Weibel, on a que PST(k) est ab\'{e}lienne et admet assez de projectifs et d'injectifs. Nous n'en n'aurons pas l'utilit\'{e}, mais il s'av\`{e}re que c'est tr\`{e}s important dans le cadre de la cat\'{e}gorie homotopique stable des schémas.\\

\textbf{Exemples}: Le faisceau des sections globales est canoniquement muni de transferts (cf plus loin). 
Les faisceaux repr\'{e}sentables fournissent \'{e}galement une classe importante de pr\'{e}faisceaux avec transferts. Si $X\in Sm/k$, on notera $\mathbb{Z}_{tr}(X)$, le faisceau r\'{e}presentable qui s'en d\'{e}duit.
Dans ce paragraphe, on rappelle quelques propri\'{e}t\'{e}s de la topologie Nisnevich dont nous aurons besoin plus tard.
\begin{definition}
Une famille de morphismes \'{e}tales $\{ p_{i}:U_{i}\rightarrow X\}$ est un recouvrement Nisnevich, si il a la propri\'{e}t\'{e} de rel\`{e}vement suivante:
Pour tout $x\in X$, il existe i et $u \in U_{i}$ tel que $p_{i}(u)=x$ et le morphisme induit $k(x)\rightarrow k(u)$ est un isomorphisme.
\end{definition}

\textbf{Exemple}: Pour illustrer le côt\'{e} arithm\'{e}tique de la topologie Nisnevich et un de ses avantages par rapport \`{a} la topologie \'{e}tale, on a l'exemple suivant. Soit k un corps de caract\'{e}ristique diff\'{e}rente  de deux, les deux morphismes $U_{0}=\mathbb{A}^{1}-\{a\}\rightarrow \mathbb{A}^{1}$ et $U_{1}=\mathbb{A}^{1}-\{0\}\stackrel{z\rightarrow z^{2}}{\rightarrow} \mathbb{A}^{1}$ forment un recouvrement Nisnevich si et seulement si $a\in (k^{*})^{2}$ et un recouvrement \'{e}tale pour tout a non nul.

\begin{lemma}
Si $\{ p_{i}:U_{i}\rightarrow X\}$ est un recouvrement Nisnevich, il existe un ouvert non vide $V\subset X$ et un indice i tel que 
$U_{i, V}\rightarrow V$ admet une section.
\end{lemma}
Pour chaque point g\'{e}n\'{e}rique x de X, il y a un point g\'{e}n\'{e}rique $u\in U_{i}$ tel que $k(x)\cong k(u)$. On a donc un morphisme birationnel entre les composantes irr\'{e}ductibles de $U_{i}$ et X, i-e $U_{i}\rightarrow X$ on a une section sur un ouvert V contenant x.\\

\textbf{Remarque}: Les points pour la topologie de Nisnevich sont les anneaux locaux hens\'{e}liens. En effet, si $\{U_{i}\rightarrow Spec(R)\}$ est un recouvrement Nisnevich, avec R hens\'{e}lien, alors un certain $U_{i}$ est fini \'{e}tale d'où $U_{i}\rightarrow Spec(R)$ est scind\'{e} et donc tout recouvrement Nisnevich admet un raffinement par le recouvrement trivial.

La raison fondamentale qui justifie en grande partie l'utilisation de la topologie \'{e}tale ou Nisnevich par rapport \`{a} la topologie de Zariski est le lemme suivant:

\begin{lemma}(Prop 6.12 [21])\\
$p:U\rightarrow X$ un recouvrement \'{e}tale (ou Nisnevich). On d\'{e}finit $\mathbb{Z}_{tr}(\check{U})$ le complexe de Cech:
\xymatrix{..\ar[rr]^{p_{0}-p_{1}+p_{2}}&&\mathbb{Z}_{tr}(U\times U)\ar[r]^{p_{0}-p_{1}}&\mathbb{Z}_{tr}(U)\ar[r] &0}.$\\$
Alors  $\mathbb{Z}_{tr}(\check{U})$ est une r\'{e}solution \'{e}tale (resp Nisnevich) de $\mathbb{Z}_{tr}(X)$, i-e  $\mathbb{Z}_{tr}(\check{U})\rightarrow  \mathbb{Z}_{tr}(X)\rightarrow ~0$ est exact.
\end{lemma}

\begin{proof} On le traite dans le cas \'{e}tale, le cas Nisnevich ne n\'{e}cessite que des modifications mineures.
Comme c'est un complexe de faisceaux, il suffit de le v\'{e}rifier en chacun des points. Comme les points en topologie étale sont les schémas henséliens strictement locaux, il suffit de  montrer que pour tout sch\'{e}ma hens\'{e}lien  local S sur k, la suite de groupes ab\'{e}liens:
\begin{center}
\xymatrix{..\ar[r]&\mathbb{Z}_{tr}(U)(S)\ar[r]^{p_{0}-p_{1}}&\mathbb{Z}_{tr}(X)(S)\ar[r] &0}(*)  est exacte.
\end{center}
Pour prouver l'exactitude, on a besoin d'une \'{e}tape suppl\'{e}mentaire. Soit Z un sous-sch\'{e}ma de $X\times S$ qui est quasi-fini sur S. On note L(Z/S) le groupe ab\'{e}lien libre engendr\'{e} par les composantes irr\'{e}ductibles de Z qui sont finies surjectives sur S. On a que L(Z/S) est fonctoriel et covariant en Z par rapport au morphismes quasi-finis sur S. Clairement la suite (*) s'obtient que la limite inductive de complexes de la forme:
\begin{center}
\xymatrix{..\ar[r] &\mathbb{Z}_{tr}(Z\times U)\ar[r] &L(Z/S)\ar[r] &0}(**) 
\end{center}
$\\$o\`{u} la limite est prise sur tous les sous-sch\'{e}mas de $X\times S$ qui sont finis surjectifs sur S. On est donc ramen\'{e} à montrer l'exactitude de (**).
Comme S est hens\'{e}lien, on a alors que Z est \'{e}galement hens\'{e}lien, comme il est fini surjectif sur S, d'ou $Z_{U}:=Z\times U\rightarrow Z$ se scinde Soit $s_{1}$ une section. On pose $Z_{U,Z}^{k}=Z_{U}\times..\times Z_{U}$. Il suffit de montrer que les $s_{k}:L((Z_{U,Z})^{k}/S)\rightarrow L((Z_{U,Z})^{k+1}/S)$ sont des homotopies contractantes où $s_{k}=L(s_{1}\times id_{(Z_{U,Z})^{k}})$.

\end{proof}

\textbf{Remarque:} En topologie de Zariski, ce lemme est faux. En effet, si $U_{i}\rightarrow X$ est un recouvrement Zariski d'un sch\'{e}ma X connexe semi-local fini sur S local; en prenant le graphe $\Gamma$ de X, on v\'{e}rifie qu'il ne provient pas d'un \'{e}lement de  $\oplus \mathbb{Z}_{tr}(U_{i})(S)$.
\begin{definition} Un faisceau Nisnevich (\'{e}tale, Zariski, ou pour n'importe quelle topologie de son choix) avec transferts est un pr\'{e}faisceau avec transferts dont le pr\'{e}faisceau sous-jacent est un faisceau Nisnevich (de même,....) sur Sm/k.
\end{definition}

\begin{proposition}
Soit F un pr\'{e}faisceau avec transferts, alors on peut faisceautiser F dans PST(k) en un faisceau Nisnevich avec transferts de mani\`{e}re unique avec la propri\'{e}t\'{e} universelle qu'on connaît.
En d'autres termes, le foncteur d'oubli de $Sh_{Nis}(SmCor/k)\rightarrow PST(k)$ admet un adjoint \`{a} droite $a_{Nis}$ qui est exact et commute avec le foncteur d'oubli des faisceaux sur Sm/k.
\end{proposition}
\begin{proof}
cf Th 13.1 [21].

\end{proof}
\subsection{Localisation dans une cat\'{e}gorie} $\\$
Soient  C une cat\'{e}gorie, S  une famille de morphismes.
\begin{definition}
Une localisation de C par S est la donn\'{e}e d'une cat\'{e}gorie $C_{S}$ et un foncteur $Q:C\rightarrow C_{S}$ satisfaisant:
\begin{enumerate}
\item 
Pour tout $s\in S$, Q(s) est un isomorphisme.
\item
Pour tout foncteur $F:\mathcal{C}\rightarrow \mathcal{A}$ tel que F(s) est un isomorphisme pour tout $s\in S$, il existe un foncteur $F_{S}:\mathcal{C}_{S}\rightarrow \mathcal{A}$ et un isomorphisme $F\cong F(s)\circ Q$,$\\$
\begin{center}
\xymatrix{ \mathcal{C} \ar[d]_{Q}\ar[r]^{F}&\mathcal{A}\\\mathcal{C}_{S}\ar@{.>}[ur]^{F_{S}}}
\end{center}
\item
Si $G_{1}$ et $G_{2}$ sont deux objets de $Fct(\mathcal{C}_{S},\mathcal{A})$ alors l'application naturelle:$\\$
$Hom_{Fct(\mathcal{C}_{S},\mathcal{A})}(G_{1},G_{2})\rightarrow Hom_{Fct(\mathcal{C},\mathcal{A})}(G_{1}\circ Q,G_{2}\circ Q)$ est un isomorphisme.

\end{enumerate}$\\$$\\$
\end{definition}
\textbf{Remarque}: (c) signifie que le foncteur $\circ Q:Fct(\mathcal{C}_{S},\mathcal{A})\rightarrow Fct(\mathcal{C},\mathcal{A})$ est pleinement fid\`{e}le. Cela implique que $F_{S}$  dans (2) est unique \`{a} unique isomorphisme pr\`{e}s.

\begin{proposition}

Si ${C}_{S}$ existe, elle est unique \`{a}  \'{e}quivalence de cat\'{e}gories pr\`{e}s.

\end{proposition}

\begin{definition}
On dit que S est un syst\`{e}me multiplicatif \`{a} droite si il verifie les axiomes ci-dessus.
\begin{enumerate}
\item
Pour tout $X\in \mathcal{C}$, $id_{X} \in \mathcal{S}$.\\
\item
Pour $f\in \mathcal{S}$, $g\in \mathcal{S}$, si $g\circ f$ existe, $g\circ f \in \mathcal{S}$\\
\item
Soient deux morphismes, $f:X\rightarrow Y et s:X\rightarrow X'$ avec $s\in \mathcal{S}$, il existe $t:Y\rightarrow Y'$ et $g:X'\rightarrow Y'$ avec $t\in \mathcal{S}$ et $g\circ s=t\circ f$. On le repr\'{e}sente par le diagramme:$\\$
\begin{center}
$\xymatrix{X'\\ X\ar[u]_{s} \ar[r]^{f}& Y}\Rightarrow \xymatrix{X'\ar@{.>}[r]_{g} & Y'\\ X\ar[u]_{s} \ar[r]^{f}& Y\ar@{.>}[u]_{t} }$
\end{center}
\item
Soient $f, g:X\rightarrow Y$ deux morphismes parall\`{e}les. Si il existe $s\in \mathcal{S}$: $W\rightarrow X$ tel que $g\circ s=f\circ s$, alors il existe $t\in \mathcal{S}:Y\rightarrow Z$ tel que $t\circ f=t\circ g$. On le r\'{e}sume par le diagramme suivant:$\\$
\begin{center}
$\xymatrix{W\ar[r]^{s} & X \ar[r]^{f}_{g} & Y\ar@{.>}[r]^{t}&Z}$
\end{center}
\end{enumerate}
\end{definition}

\textbf{Remarque}: On peut  d\'{e}finir de mani\`{e}re analogue un syst\`{e}me multiplicatif \`{a} gauche.
\begin{definition}
Soit S un syst\`{e}me multiplicatif \`{a} droite et $X\in \mathcal{C}$. On d\'{e}finit la cat\'{e}gorie $S^{X}$ comme suit:
$Ob(S^{X})= \{s: X\rightarrow X' /s\in \mathcal{S}\} $
$\\$
$Hom((s:X\rightarrow X'),(s:X\rightarrow X''))= \{h: X\rightarrow X'' /h\circ s=s'\}$.
\end{definition}
\begin{proposition}
Si S est un syst\`{e}me multiplicatif \`{a} droite, alors la cat\'{e}gorie $S^{X}$ est filtrante.
\end{proposition}
La preuve est facile et laiss\'{e}e en exercice.

\begin{definition}
Pour S un syst\`{e}me multiplicatif \`{a} droite et X,Y dans $Ob(\mathcal{C})$.$\\$
 On consid\`{e}re la categorie $\mathcal{C}[S^{-1}]$ avec les mêmes objets que $\mathcal{C}$ et comme groupes de morphismes: $Hom_{\mathcal{C}_{S^{-1}}}(X,Y)=\varinjlim_{Y'}Hom_{\mathcal{C}}(X,Y')$.
\end{definition}$\\$

On note $Q:\mathcal{C}\rightarrow\mathcal{C}[S^{-1}]$ le foncteur canonique.

\begin{theorem}(Verdier)\\
(i) $\mathcal{C}[S^{-1}]$ est une cat\'{e}gorie additive munie d'un automorphisme T.\\
(ii) $\mathcal{C}[S^{-1}]$ est une cat\'{e}gorie triangul\'{e}e où un triangle T est distingu\'{e} si T est isomorphe \`{a} l'image par Q d'un triangle distingu\'{e} de $\mathcal{C}[S^{-1}]$.\\
(iii) La cat\'{e}gorie $\mathcal{C}[S^{-1}]$ v\'{e}rifie la propri\'{e}t\'{e} universelle suivante:\\
Pour tout foncteur $F:\mathcal{C}\rightarrow\mathcal{A}$ tel que F(s) est un isomorphisme pour s dans S, on a la factorisation suivante:
\xymatrix{ \mathcal{C}\ar[d]^{F}\ar[r]^{Q}&\mathcal{C}[S^{-1}] \\ \mathcal{A}\ar@{.>}[ur]}

\end{theorem}

\subsection{Cat\'{e}gorie triangul\'{e}e de motifs mixtes g\'{e}om\'{e}triques} $\\$
On consid\`{e}re $K^{b}({SmCor/k})$ la cat\'{e}gorie homotopique born\'{e}e. On va localiser par rapport \`{a} une certaine classe de complexes pour avoir la propri\'{e}t\'{e} de Mayer-Vietoris et l'invariance homotopique. On consid\`{e}re donc la classe T de complexes de la forme:$\\$

\begin{center}
(1) $[X\times \mathbb{A}^{1}]\rightarrow[X]$ pour X lisse s\'{e}par\'{e} sur k.
\end{center}
 
\begin{center}
(2) $[U\cap V]\stackrel{j_{U}\oplus j_{V}}{\rightarrow}[U]\oplus[V]\stackrel{j_{U}-j_{V}}{\rightarrow}[X]$
\end{center}
 avec U, V recouvrement ouvert de X comme ci-dessus.$\\$

On note alors $DM_{gm}^{eff}(k)$ la cat\'{e}gorie obtenue en quotientant $K^{b}({SmCor/k})$ par la sous cat\'{e}gorie \'{e}paisse $\bar{T}$ minimale contenant T et en prenant ensuite l'enveloppe pseudo-ab\'{e}lienne. On remarquera qu'en ce qui concerne la structure triangul\'{e}e, il n'est pas \'{e}vident a priori que l'enveloppe pseudo-ab\'{e}lienne le soit aussi, mais c'est pourtant vrai d'apr\`{e}s  Balmer-Schlichting [3]. On note $M_{gm}$ le foncteur de Sm/k vers $DM_{gm}^{eff}(k)$.
On a le lemme trivial suivant:
\begin{lemma} Dans $DM_{gm}^{eff}(k)$, on a le triangle distingu\'{e} suivant dit de Mayer-Vietoris:$\\$

\begin{center}
(1) $M_{gm}([U\cap V])\stackrel{j_{U}\oplus j_{V}}{\rightarrow}M_{gm}([U]\oplus[V])\stackrel{j_{U}-j_{V}}{\rightarrow}M_{gm}([X])\rightarrow M_{gm}([U\cap V][1])$.
\end{center}
\end{lemma}$\\$

La structure tensorielle sur SmCor/k s'\'{e}tend naturellement \`{a} $K^{b}({SmCor/k})$ et descend \`{a} $DM_{gm}^{eff}(k)$ par propri\'{e}t\'{e} universelle de la localisation. On a donc une structure naturelle de  cat\'{e}gorie monoïdale sym\'{e}trique sur $DM_{gm}^{eff}(k)$ et qui fait de $M_{gm}$ un foncteur tensoriel. $\\$
On notera donc au passage, la remarquable identit\'{e} dite de Künneth:$\\$ 
\begin{center}
$M_{gm}([X])\otimes M_{gm}([Y])=M_{gm}([X\times Y])$.
\end{center}
$\\$
On a l'objet unit\'{e} $M_{gm}(Spec(k))$ que l'on notera $\mathbb{Z}$. $\\$On pose $ \mathbb{Z}(1):=\tilde{M}_{gm}(\mathbb{P}^{1})[-2]$ o\`{u} $\tilde{M}_{gm}(X)$ d\'{e}signe le motif r\'{e}duit de X repr\'{e}sent\'{e} par\\ $[X]\rightarrow[Spec(k)]$ dans $K^{b}({SmCor/k})$. On l'appelle le motif de Tate  et on note $\mathbb{Z}(n)$ sa n-i\`{e}me puissance tensorielle. De m\^{e}me, on pose $A(n)=A\otimes \mathbb{Z}(n)$.$\\$

Pour obtenir la d\'{e}finition des motifs  sur k, il ne nous reste plus qu'\`{a} inverser le motif de Tate. On d\'{e}finit donc $DM_{gm}(k)$ la cat\'{e}gorie dont les objets sont les couples (A,n) avec A dans $DM_{gm}^{eff}(k)$ et $n\in \mathbb{Z}$. $\\$
De plus, $Hom((A,n),(B,m))= \varinjlim_{k}Hom_{DM_{gm}^{eff}(k)}(A(k+n),B(k+m))$.\\

Cette fois-ci la structure tensorielle sur $DM_{gm}(k)$ est moins \'{e}vidente et il est faux en g\'{e}n\'{e}ral que l'on conserve la structure tensorielle en inversant un objet Q, cependant c'est effectivement le cas si l'involution de permutation sur $Q\otimes Q$ est l'identit\'{e}. Pour obtenir ce r\'{e}sultat, on rappelle certaines propri\'{e}t\'{e}s des motifs purs.

\begin{definition} On note $Chow^{eff}(k)$ la cat\'{e}gorie des motifs de Chow dont les objets sont les k-sch\'{e}mas lisses projectifs et les morphismes donn\'{e}s par $CH^{dim X}(X\times Y)$ cycles de dimension d dans $X\times Y$(si X connexe sinon, faire la somme sur les composantes connexes) avec comme loi de composition la même que  celle des correspondances finies. On note $Chow^{eff}(k)$, l'enveloppe pseudo-ab\'{e}lienne. On note Chow le foncteur de SmProj/k vers $Chow^{eff}(k)$.
\end{definition}

\begin{proposition} On a le diagramme commutatif suivant:$\\$
\begin{center}
\xymatrix{SmProj/k\ar[r]\ar[d]^{Chow}& Sm/k \ar[d]^{M_{gm}}\\Chow^{eff}(k)\ar[r]&DM_{gm}^{eff}(k)}
\end{center}

\end{proposition}

\begin{proof} Il faut montrer que pour X, Y des vari\'{e}t\'{e}s projectives lisses,  il y a un morphisme canonique:$\\$
\begin{center}
$CH^{dim X}(X\times Y)\rightarrow Hom_{DM_{gm}^{eff}(k)}({M}_{gm}(X),{M}_{gm}(Y))$.
\end{center}$\\$
On note $h_{0}(X,Y)$ le conoyau du morphisme $Cor(X\times \mathbb{A}^{1},Y)\rightarrow Cor(X,Y)$ donn\'{e} par la diff\'{e}rence des restrictions de $X\times{0}$ et $X\times{1}$. $\\$On a alors que le morphisme \'{e}vident $Cor(X,Y)\rightarrow Hom_{DM_{gm}^{eff}(k)}({M}_{gm}(X),{M}_{gm}(Y))$ se factorise par $h_{0}(X,Y)$.
Or, on a d'apr\`{e}s [11] un morphisme canonique $Cor(X,Y)\rightarrow CH^{dim X}(X\times Y)$ qui est un isomorphisme.
\end{proof}

On d\'{e}duit donc de la proposition ci-dessus et du fait correspondant dans les motifs de Chow que l'involution de permutation est l'identit\'{e}. On a donc obtenu le corollaire suivant.
$\\$
$DM_{gm}(k)$ est une cat\'{e}gorie triangul\'{e}e tensorielle sur k.

On a la th\'{e}or\`{e}me naturel sur $DM_{gm}^{eff}(k)$:

\begin{theorem}(Voevodsky, [32],4.3.1)\\
(i) Pour tous M, N dans $DM_{gm}^{eff}(k)$, l'homomorphisme canonique:

\begin{center}
$Hom_{DM_{gm}^{eff}(k)}(M,N)\rightarrow Hom_{DM_{gm}^{eff}(k)}(M(1),N(1))$
\end{center}

est un isomorphisme.$\\$

(ii) Le foncteur $i:DM_{gm}^{eff}(k)\rightarrow DM_{gm}(k)$ est un foncteur pleinement fid\`{e}le.
\end{theorem}

On a \'{e}galement une formule pour le fibr\'{e} projectif.
\begin{proposition}
Soit X un k-schéma lisse, $\mathcal{E}$ un fibr\'{e} vectoriel sur X, on a un isomorphisme canonique dans $DM_{gm}^{eff}(k)$:$\\$
\begin{center}
$M_{gm}(\mathbb{P}(E))=\oplus_{n\geq \text{dim E}-1} M_{gm}(X)(n)[2n]$
\end{center}
\end{proposition}

\begin{proof}
On suppose $d=\text{dim E}>0$. Soit $\mathcal{O}(1)$ le faisceau inversible sur $\mathbb{P}(E)$. Par [31], Cor 3.4.3, il d\'{e}finit un morphisme $\tau_{1}:M_{gm}(\mathbb{P}(E))\rightarrow \mathbb{Z}(1)[2]$.
Pour $n\geq 0$, on pose $\tau_{n}$ la composition:\\
\begin{center}
$M_{gm}(\mathbb{P}(E))\stackrel{M_{gm}(\Delta)}{\rightarrow}M_{gm}(\mathbb{P}(E)^{n})=M_{gm}(\mathbb{P}(E)^{\otimes n})\rightarrow \mathbb{Z}(n)[2n]$.
\end{center}

Soit alors $\sigma_{n}$ la composition:\\
\begin{center}
$M_{gm}(\mathbb{P}(E))\stackrel{M_{gm}(\Delta)}{\rightarrow}M_{gm}(\mathbb{P}(E))\otimes M_{gm}(\mathbb{P}(\mathcal{E}))\stackrel{q\otimes \tau_{n}}{\rightarrow}M_{gm}(X)(n)[2n]$.
\end{center}
$\\$
o\`{u} $q:M_{gm}(\mathbb{P}(E))\rightarrow M_{gm}(X)$ induit par le morphisme stuctural.
On a alors un morphisme:$\\$
\begin{center}
$\Sigma=\oplus{\sigma_{n}}:M_{gm}(\mathbb{P}(E))\rightarrow \oplus_{n\geq \text{dim E}-1} M_{gm}(X)(n)[2n].$
\end{center}
$\\$$\\$
Montrons que c'est un isomorphisme. D'abord $\Sigma$ est fonctoriel par rapport \`{a} X. En faisant une r\'{e}currence sur le nombre d'ouverts dans un recouvrement trivialisant de E, on se ram\`{e}ne au cas o\`{u} le fibr\'{e} vectoriel est trivial. Dans ce cas, $\Sigma$ est le produit fibr\'{e} du fibr\'{e} trivial sur Spec(k) et $Id_{M_{gm}(X)}$. On a donc \`{a} le faire pour le cas du spectre d'un corps et dans ce cas-l\`{a}, c'est comme dans [32].

\end{proof}

\subsection{Extension aux vari\'{e}t\'{e}s singuli\`{e}res}$\\$
Ce paragraphe ne nous sera pas utile pour la suite, n\'{e}anmoins, il indique comment on passe de la cat\'{e}gorie des k-sch\'{e}mas lisses aux k-sch\'{e}mas quelconques  modulo la r\'{e}solution des singularit\'{e}s pour le corps de base k, ce qui concerne donc que  la caract\'{e}ristique z\'{e}ro pour le moment, et peut-être en toute caract\'{e}ristique avec les travaux r\'{e}cents de Kawanoue et Matsuki [18], [19].$\\$
L'objectif est d'\'{e}tendre le foncteur $M_{gm}:Sm/k\rightarrow DM_{gm}^{eff}(k)$ en un foncteur \\ $M_{gm}:~Var/k\rightarrow~DM_{gm}^{eff}(k)$.
Le foncteur en question v\'{e}rifiera les propri\'{e}t\'{e}s suivantes:\\

(1) $M_{gm}([X])\otimes M_{gm}([Y])=M_{gm}([X\times Y])$ (structure tensorielle).
$\\$$\\$
(2) $M_{gm}([X\times \mathbb{A}^{1}])=M_{gm}([X])$ (invariance homotopique).
$\\$$\\$
(3) $M_{gm}([U\cap V])\stackrel{j_{U}\oplus j_{V}}{\rightarrow}M_{gm}([U]\oplus[V])\stackrel{j_{U}-j_{V}}{\rightarrow}M_{gm}([X])\rightarrow M_{gm}([U\cap V][1])$ (Mayer-Vietoris).
$\\$$\\$
(4) $M_{gm}(p^{-1}(Z))\rightarrow M_{gm}(Z)\oplus M_{gm}(X_{Z})\rightarrow M_{gm}(X)\rightarrow M_{gm}(p^{-1}(Z))[1]$ (éclatement).
$\\$$\\$
(5) $M_{gm}(\mathbb{P}(\mathcal{E}))=\oplus M_{gm}([X])(n)[2n]$(formule pour le fibré projectif).
$\\$$\\$
On  a \'{e}galement des suite de Gysin, des motifs \`{a} supports compacts et une th\'{e}orie de la dualit\'{e} associ\'{e}e.

On note par $DM_{-}^{eff}(k)$ la sous-cat\'{e}gorie pleine de $D^{-}(Sh_{Nis}(SmCor/k))$ qui consiste en les complexes qui sont homotopiquement invariants.
Commençons par rappeler quelques propri\'{e}t\'{e}s g\'{e}n\'{e}rales des pr\'{e}faisceaux avec transferts. Soit $\Delta^{\bullet}$  l'objet cosimplicial standard de Sm/k. Pour tout pr\'{e}faisceau F sur Sm/k, soit $C_{*}(F)$ le complexe de pr\'{e}faisceaux de la forme $C_{n}(F)(X)=F(X\times \Delta^{n})$ avec les diff\'{e}rentielles donn\'{e}es par les sommes altern\'{e}es des morphismes qui correspondent aux cod\'{e}g\'{e}n\'{e}rescences de $\Delta^{n}$. On a facilement que si F est un pr\'{e}faisceau avec transferts pour sa topologie favorite alors il en est de même pour son complexe simplicial associ\'{e}.
On note $h_{i}^{Nis}(F)$ les faisceaux de cohomologie $H^{-i}(C_{*}(F))$.
 
 \begin{proposition}
 Pour tout faisceau avec transferts F sur k, les faisceaux  $h_{i}^{Nis}(F)$ sont invariants par homotopie.
 \end{proposition}
 
 \begin{proof}
 
 cf Th 13.8,[21]. 
 
 \end{proof}

  Soit X un k-sch\'{e}ma de type fini. On  d\'{e}finit le pr\'{e}faisceau $L(X):Sm/k\rightarrow Ab$ avec L(X)(U) le groupe ab\'{e}lien libre engendr\'{e} par les composantes irr\'{e}ductibles  Z de $X\times U$ qui sont  quasi-finies  sur U et dominantes sur une composante connexe de U.
Pour X un k-sch\'{e}ma de type fini, on \'{e}crit $C_{*}(X)$ \`{a} la place de $C_{*}(L(X))$. D'apr\`{e}s la proposition ci-dessus,
on a alors un foncteur $L: Var/k\rightarrow~ DM^{eff}_{-}(k)$. Nous allons voir que ce foncteur se factorise par $DM^{eff}_{gm}(k)$.

\begin{proposition} Soit X un k-sch\'{e}ma de type fini, on suppose $X=U\cup V$ avec U, V ouverts. Alors on a un triangle distingu\'{e} de Mayer-Vietoris:$\\$$\\$
$C_{*}(U\cap V)\rightarrow C_{*}(U)\oplus C_{*}(V)\rightarrow C_{*}(X)\rightarrow C_{*}(U\cap V)[1]$

\end{proposition}

\begin{proof}
Il nous suffit de v\'{e}rifier que l'on a suite exacte de faisceaux Nisnevich:\\
\begin{center}
$0\rightarrow L(U\cap V)\rightarrow L(U)\oplus L(V)\rightarrow L(X)\rightarrow L(U\cap V)[1]\rightarrow 0$
\end{center}

\end{proof}

On aura besoin du résultat suivant donn\'{e} dans sous une forme l\'{e}g\`{e}rement diff\'{e}rente dans [31] et d\'{e}montr\'{e}e dans [32].

\begin{theorem}
Soit k un corps admettant la r\'{e}solution des singularit\'{e}s et F un pr\'{e}faisceau avec transferts sur Sm/k tel que tout pour sch\'{e}ma lisse X et toute section $\phi\in F(X)$ il y a un morphisme propre birationnel $p:X'\rightarrow X$ avec $F(p)(\phi)=0$ alors le complexe $C_{*}(F)$ est quasi-isomorphe \`{a} z\'{e}ro.

\end{theorem}

On a alors le corollaire suivant qui fournit le triangle distingu\'{e} d'\'{e}clatement:

\begin{proposition}
Soit un carr\'{e} cart\'{e}sien entre k-sch\'{e}mas de type fini de la forme:\\
\begin{center}
\xymatrix{p^{-1}(Z)\ar[d]\ar[r]& X_{Z}\ar[d]^{p}\\Z\ar[r]&X}
\end{center}
avec les conditions suivantes:\\
\begin{center}
(i) $p: X_{Z}\rightarrow X$ est propre et le morphisme $Z\rightarrow X$ est une immersion ferm\'{e}e.
\end{center}

(ii) Le morphisme $p^{-1}(X-Z)\rightarrow X$ est un isomorphisme. Alors il y a un triangle distingu\'{e} dans $DM^{eff}_{-}(k)$ de la forme:\\
\begin{center}
$C_{*}(p^{-1}(Z))\rightarrow C_{*}(Z)\oplus C_{*}(X_(Z))\rightarrow C_{*}(X)\rightarrow C_{*}(p^{-1}(Z))[1]$.
\end{center}
\end{proposition}

\begin{proof}
Il suffit de voir que la suite de pr\'{e}faisceaux:\\
 
\begin{center}
$0\rightarrow L(p^{-1}(Z))\rightarrow L(Z)\oplus L(X_{Z})\rightarrow L(X)$
\end{center}
 est exacte et que le pr\'{e}faisceau quotient v\'{e}rifie $L(X)/ (L_{*}(Z)\oplus L_{*}(X_{Z}))$ les conditions du th\'{e}or\`{e}me.
\end{proof}

On a alors comme corollaire que pour tout k-sch\'{e}ma de type fini X, k admettant la r\'{e}solution des singularit\'{e}s $C_{*}(X)$ est dans  $DM^{eff}_{gm}(k)$.

On a de plus les propositions suivantes qui d\'{e}coulent formellement de la proposition ci-dessus et du cas projectif lisse et du triangle distingu\'{e} d'\'{e}clatement.
\begin{proposition}

(i) $ C_{*}(X)\otimes C_{*}(Y)=C_{*}(X\times Y)$.

$\\$
(ii) $ C_{*}(X\times\mathbb{A}^{1})= C_{*}(X)$.

\end{proposition}
 
 On veut d\'{e}crire maintenant les groupes de morphismes $Hom(C_{*}(X),C_{*}(F))$ pour X un k-sch\'{e}ma de type fini et F un faisceau Nisnevich avec transferts.
 Dans le cas de X sch\'{e}ma lisse, on a d'apr\`{e}s [31] qu'il est isomorphe \`{a} $\mathbb{H}_{Nis}(X,C_{*}(F))$, mais il se trouve que dans le cas singulier, contrairement  a ce qu'on pourrait penser, on va obtenir quelque chose de diff\'{e}rent.
 Il nous faut donc une nouvelle topologie de Grothendieck, dite cdh qui va nous donner les r\'{e}sultats voulus.
 
 \begin{definition}
 La topologie cdh sur Sch/K est la topologie de Grothendieck engendr\'{e} par la pr\'{e}-topologie suivante:$\\$

(1) les recouvrements Nisnevich,

\begin{center}
 (2) les recouvrements de la forme $X'\coprod Z\stackrel{p\coprod i}{\rightarrow} $ avec p un morphisme propre et i une immersion ferm\'{e}e et tel que $p^{-1}(X-i(Z))\rightarrow X-i(Z)$ est une isomorphisme.
\end{center}
 \end{definition}
 On note $\pi$ le morphisme de sites canonique vers le  site Nisnevich.
 Le th\'{e}or\`{e}me suivant se d\'{e}duit du th\'{e}or\`{e}me 3 ci-dessus:
 
 \begin{theorem}[32]
 X un k-sch\'{e}ma de type fini et F un pr\'{e}faisceau avec transferts sur Sm/k. Alors pour tout $i\geq 0$ il y a des isomorphismes canoniques:\\
 
\begin{center}
$Hom(C_{*}(X),C_{*}(F)[i])=\mathbb{H}_{cdh}^{i}(X,C_{*}(\pi^{*}F))=\mathbb{H}_{cdh}^{i}(X,\pi^{*}(C_{*}(F))).$
\end{center}
$\\$
 En particulier, si X est lisse on a un isomorphisme canonique:\\
\begin{center}
 $\mathbb{H}_{cdh}^{i}(X,C_{*}(\pi^{*}F))=\mathbb{H}_{cdh}^{i}(X,\pi^{*}(C_{*}(F)))=\mathbb{H}_{Nis}^{i}(X,C_{*}(F))$
\end{center}
 
 \end{theorem}$\\$
 
 On en d\'{e}duit le corollaire suivant:
 $\\$
\textbf{ Corollaire}:
Si k admet la r\'{e}solution des singularit\'{e}s et X un k-sch\'{e}ma de type fini. Soit E un fibr\'{e} vectoriel sur X. On note $p:\mathbb{P}(E)\rightarrow X$ le fibr\'{e} projectif associ\'{e}. Alors un isomorphisme canonique dans $DM^{eff}_{-}(k)$ de la forme:\\
\begin{center}
$C_{*}(\mathbb{P}(E))=\oplus_{n=0}^{\text{dim E}-1} C_{*}(X)(n)[2n]$.
\end{center}

\begin{proof}
Le th\'{e}or\`{e}me pr\'{e}c\'{e}dent nous donne un morphisme canonique:\\

\begin{center}
$C_{*}(\mathbb{P}(E))=\oplus_{n=0}^{\text{dim E}-1} C_{*}(X)(n)[2n].$
\end{center}
 L'isomorphisme suit alors  de la r\'{e}solution des singularit\'{e}s, du triangle distingu\'{e} pour un \'{e}clatement et la proposition 8.

\end{proof}

\section{Localisation des correspondances finies}$\\$
Les d\'{e}monstrations de toutes les propositions de ce paragraphe sont  tr\`{e}s techniques et fastidieuses, on renvoit donc \`{a} [15bis].
Soit X un sch\'{e}ma. On notera $X_{x}^{h}:=Spec (\mathcal{O}^{h}_{X,x})$ o\`{u} $\mathcal{O}_{X,x}^{h}$ d\'{e}signe l'hens\'{e}lis\'{e} de $\mathcal{O}_{X,x}$ et 
 $X^{h}=\coprod_{x\in X} {X_{x}^{h}}$. $\\$Et on note $l^{h}_{x}$ le morphisme de $X_{x}^{h}$ vers X.
On a un morphisme naturel de:
\begin{center}
$\oplus Cor(\mathcal{Y},X_{x}^{h})\stackrel{\sum_{x\in X}[l^{h}_{x}]\circ -}{\rightarrow}Cor(\mathcal{Y},X)$     (1)
\end{center}

On a la proposition suivante:
\begin{lemma}(Corollaire 2.12 [15 bis])
X un S-sch\'{e}ma, $\mathcal{Y}$ un sch\'{e}ma local hens\'{e}lien. Il existe un morphisme canonique:\\
\begin{center}
$ Cor(\mathcal{Y},X)\stackrel{\sigma_{\mathcal{Y},X}}{\rightarrow}\oplus_{x\in X} Cor(\mathcal{Y},X_{x}^{h}) (2)$
\end{center}
 $\\$ avec les propri\'{e}t\'{e}s suivantes:\\
(a) $\sigma_{\mathcal{O},X^{h}}$ est une section de (1) telle que le carr\'{e} suivant:\\
\begin{center}
\xymatrix{Cor(\mathcal{O},X)\ar[d]^{Cor(\alpha,X)}\ar[r]^{\sigma_{\mathcal{O},X^{h}}}&\oplus Cor(\mathcal{O},X_{x}^{h})\ar[d]^{Cor(\alpha,X_{x}^{h})}\\Cor(\mathcal{O}',X)\ar[r]^{\sigma_{\mathcal{O}',X^{h}}}&\oplus Cor(\mathcal{O}',X_{x}^{h})}
\end{center}
 soit commutatif pour tout sch\'{e}ma local hens\'{e}lien $\mathcal{O}'$ et toute correspondance finie $\alpha\in ~Cor(\mathcal{O}',\mathcal{O})$.\\

(b) Pour $g:\mathcal{Y}\rightarrow X$ un S-morphisme suivant le point x de l'image [g] par (2) est donn\'{e} par:
\begin{center}
$\sigma_{\mathcal{O},X}([g])_{x}=\begin{cases}$
$[\bar{g}]$ \text{si} $x=\tau$ $\\ $0  \text{sinon}
$
\end{cases}$
\end{center}
$\tau$ \'{e}tant l'image du point ferm\'{e} de $\mathcal{Y}$ et $\bar{g}$ le morphisme d\'{e}duit de g.\\

\xymatrix{Spec(\mathcal{Y})\ar@/^2pc/[rr]^{g}\ar[r]^{\bar{g}} &X_{x}^{h} \ar[r]^{l_{X,x}^{h}}&X}

\end{lemma}

On en d\'{e}duit donc qu'\'{e}tant donn\'{e}e une correspondance $\alpha \in Cor(X,Y)$, \`{a} chaque point x (resp~ y) de X (resp Y), on a une correspondance finie $\alpha_{x,y}$ de $Cor(X^{h}_{x},Y^{h}_{y})$ donn\'{e}e par:\\
$\alpha_{x,y}:=\sigma_{X^{h}_{x},Y}(\alpha\circ [l^{h}_{x}])_{y}$. On obtient alors la d\'{e}composition locale suivante:
\begin{center}
$\alpha\circ [l^{h}_{X,x}]=\sum_{y\in Y}{[l^{h}_{Y,y}]\circ \alpha_{x,y}}$
\end{center}

Dans cette proposition, on \'{e}tudie le comportement de cette d\'{e}composition par rapport \`{a} la composition et au produit tensoriel des correspondances.

\begin{proposition}(Prop 2.12)
\begin{center}
(i) Soit $\alpha$, $\beta$ dans $Cor(X,Y)$ et $Cor(Y,Z)$, on a pour tout point $x\in X$ et $z\in Z$ l'\'{e}galit\'{e}:
$(\beta\circ\alpha)_{x,z}=\sum_{y\in Y}{\beta_{y,z}\circ\alpha_{x,y}}$.
\end{center}

(ii) Soient X, Y, X', Y' des S-sch\'{e}mas, $\alpha \in Cor(X,X')$ et $\beta \in Cor(Y,Y')$. Pour tout point e de $X\times Y$, tout point x' de X' et y' de Y'.\\
Par la propri\'{e}t\'{e} universelle de l'hens\'{e}lisation, \'{e}tant fix\'{e} un point e de $X\times Y$ d'image x et y, on a une factorisation:$\\$

\begin{center}
\xymatrix{(X\times Y)^{h}_{e}\ar[d]\ar[r]^{m_{X,Y,e}}&X_{x}^{h}\times Y^{h}_{y}\ar[d]\\(X\times Y)^{h}\ar[r]^{m_{X,Y}}&X\times Y}
 o\`{u} $m_{X,Y}$ est donn\'{e}e par le diagramme suivant:\\
 \xymatrix{(X\times Y)^{h}\ar[dr]^{l^{h}_{X\times Y}}\ar[r]^{m_{X,Y}}&X^{h}\times Y^{h}_{y}\ar[d]^{l^{h}_{X}}\times l^{h}_{Y}\\&X\times Y}
\end{center}
$\\$
 on a alors l'\'{e}galit\'{e}:
\begin{center}
$(\alpha_{x,x'}\otimes\beta_{y,y'})\circ[m_{X,Y,e}]=\sum_{e' }{[m_{X',Y',e}]\circ(\alpha\otimes\beta)_{e,e'}}$
\end{center}
$\\$
 dans  $Cor((X\times Y)_{e}^{h},X^{'h}_{x'}\times Y_{e}^{'h})$ o\`{u}  e\'{} point de $X'\times Y'$ d'image $(x',y')$.
\end{proposition}

Etant donn\'{e} un S-sch\'{e}ma X et un syst\`{e}me projectif de S-sch\'{e}mas $(U_{\lambda})_{\lambda \in \Lambda}$ avec $\Lambda$ une cat\'{e}gorie filtrante. On pose $\mathcal{U}=\varprojlim U_{\lambda}$ et 
$Cor\{\mathcal{U},X):=\varinjlim_{\lambda\in\Lambda^{op}}Cor(U_{\lambda},X)$. Si F est un pr\'{e}faisceau, on d\'{e}finit $F\{\mathcal{U}\}:=\varinjlim_{\lambda\in\Lambda^{op}}F(U_{\lambda})$.

\begin{proposition}(Prop 2.7)
Etant donn\'{e} un S-sch\'{e}ma X et un syst\`{e}me projectif de S-schémas $(U_{\lambda})_{\lambda \in \Lambda}$  de limite projective $\mathcal{U}$.\\
 Il existe d'uniques morphismes:$\\$
\begin{center}
$\{\sigma\}^{h}_{\mathcal{U},X,n}:\mathbb{Z}_{tr}[(X^{h})^{n}_{X}]\{\mathcal{U}\}\rightarrow \mathbb{Z}_{tr}[(X^{h})^{n+1}_{X}]\{\mathcal{U}\}$ $n\geq 0$
\end{center}
 satisfaisant les propri\'{e}t\'{e}s suivantes:
$\\$
(i) (Homotopie) Pour tout n, on a les relations:
\begin{center}
$d_{n+1}\circ\{\sigma\}^{h}_{\mathcal{U},X,n}+\{\sigma\}^{h}_{\mathcal{U},X,n-1}\circ d_{n}=id$.
\end{center}

(ii) On a un diagrammme commutatif:$\\$
\begin{center}
\xymatrix{\mathbb{Z}_{tr}[(X^{h})^{n}_{X}]\{\mathcal{U}\}\ar[d]
\ar[r]^{\{\sigma\}^{h}_{\mathcal{U},X,n}}&\mathbb{Z}_{tr}[(X^{h})^{n+1}_{X}]\{\mathcal{U}\}\ar[d]\\\mathbb{Z}_{tr}[(X^{h})^{n}_{X}](\mathcal{U})\ar[r]^{\sigma^{h}_{\mathcal{U},X,n}}&\mathbb{Z}_{tr}[(X^{h})^{n+1}_{X}](U)}
\end{center}

\end{proposition}

Avec les mêmes notations que ci-dessus, on obtient que pour $\alpha\in c_{S}\{\mathcal{U},X\}$, il existe un unique \'{e}l\'{e}ment $\{\alpha\}_{x}$ de $c_{S}\{\mathcal{U},X)$ verifiant:\\
\begin{center}
$\alpha:=\sum{[l^{h}_{X,x}]\circ\{\alpha\}_{x}}$
\end{center}
D\'{e}signons par $\{l\}^{h}_{X,x}$ l'\'{e}lement de $c_{S}\{X^{h}_{x},X)$ induit par les morphismes $U\rightarrow X$ pour U voisinage Nisnevich de x. Une correspondance finie $\alpha\in c_{S}(X,Y)$   admet une d\'{e}composition locale raffin\'{e}e de la forme:\\
\begin{center}
$\alpha\circ\{l\}^{h}_{X,x}=\sum{[l^{h}_{Y,y}]\circ\{\alpha\}_{x,y}}$
\end{center}
 où $\{\alpha\}_{x,y}\in Cor\{X^{h}_{x},Y^{h}_{y})$.

On a les propositions analogues pour la d\'{e}composition raffin\'{e}e.
\begin{proposition}(Prop 2.17)
$\\$
(i) Soit $\alpha$, $\beta$ dans $Cor(X,Y)$ et $Cor(Y,Z)$, on a pour tout point $x\in X$ et $z\in Z$ l'\'{e}galit\'{e}:
\begin{center}
$\{\beta\circ\alpha\}_{x,z}=\sum_{y\in Y}{\beta_{y,z}\circ\{\alpha\}_{x,y}}$.
\end{center}

$\\$

(ii) Soient X, Y, X', Y' des S-sch\'{e}mas,$\alpha \in Cor(X,X')$ et $\beta \in c_{S}(Y,Y')$. Pour tout point e de $X\times Y$, tout point x' de X' et y' de Y'.

 On a alors l'\'{e}galit\'{e}:\\
\begin{center}
$(\{\alpha\}_{x,x'}\otimes\{\beta\}_{y,y'})\circ[m_{X,Y,e}]=\sum_{e'}[m_{X',Y',e}]\circ\{\alpha\otimes\beta\}_{e,e'}$
\end{center}
o\`{u} e' point de $X'\times Y'$ d'image $(x',y')$ dans  $Cor((X\times Y)_{e}^{h},X^{'h}_{x'}\times Y_{y'}^{'h})$.
  \end{proposition}

\section{R\'{e}solution de Godement}
Dans la suite de cette section nous fixons une cat\'{e}gorie $S = Sch/S$, $Var/S$ ou $Sm/S$,
tous les pr\'{e}faisceaux sont d\'{e}finis sur S et nous renvoyons \`{a} la  d\'{e}finition 1.10
pour la notion de faisceaux avec transferts.$\\$ A nouveau, on renvoie \`{a} Ivorra [15 bis].
Rappelons qu'une monade est la donn\'{e}e d'un endofoncteur M dans une cat\'{e}gorie C et de transformations naturelles $\mu:MM\rightarrow M$ et $\eta:id\rightarrow M$ avec les diagrammes commutatifs suivants:$\\$

\begin{center}

\xymatrix{M\ar[r]^{\eta M}\ar[rd]^{id}&MM\ar[d]^{\mu}&M\ar[ld]^{id}\ar[l]_{M\eta}\\&M}
\end{center}
\begin{center}
\xymatrix{MMM\ar[d]^{\mu M}\ar[r]^{M\mu}&MM\ar[d]^{\mu}\\MM\ar[r]^{\mu}&M}
\end{center}

$\\$
A partir d'une monade $(M,\mu,\eta)$, \`{a} chaque objet C de $\mathcal{C}$, on construit un objet cosimplicial $B^{\bullet}(M,C)$ de $\mathcal{C}$ avec une coaugmentation de C dans ce dernier. $\\$Les n-cosimplexes sont donn\'{e}s par $ M^{n+1}C$ de $\mathcal{C}$, les cofaces par les morphismes:
\begin{center}
$\delta^{n-1}_{i}:=M^{i}\eta M^{n-i}:M^{n}C\rightarrow M^{n+1}C$       $i=0...n$
\end{center}
\text{et les cod\'{e}g\'{e}n\'{e}rescences par}:$\\$
\begin{center}
$\sigma^{n}_{i}:=M^{i}\eta M^{n-1-i}:M^{n+1}C\rightarrow M^{n}C$      $i=0...n-1$.
\end{center}

Nous allons maintenant appliquer les r\'{e}sultats concernant la localisation Nisnevich des correspondances finies que nous avons obtenus pr\'{e}cédemment.
 Dans la suite, nous d\'{e}signons par $s_{x}$ la composante suivant le point x d'un
el\'{e}ment s du produit 
\begin{center}
$\prod _{x\in X}F\{X_{x}^{h}\}$. 
	
\end{center}

Par d\'{e}finition $F\{X_{x}^{h}\}$
 correspond \`{a} la fibre Nisnevich de F au point x.
D\'{e}signons  $\mathcal{G}_{Nis}$ la monade de la cat\'{e}gorie des faisceaux Nisnevich d\'{e}finie de la mani\`{e}re suivante:
$\\$

(i) Les sections sur X du faisceau $G_{Nis}F$ sont donn\'{e}s par:\\
\begin{center}
$\mathcal{G}_{Nis}F(X)=\prod_{x\in X} F\{X_{x}^{h}\}$
\end{center}

les morphismes structuraux \'{e}tant d\'{e}termin\'{e}s par les \'{e}galit\'{e}s:\\
$[\mathcal{G}_{Nis}F(p)(s)]_{x}=F\{\{p\}_{x}^{h}\}(s_{p(x)})$

avec x un point de X, $p:X\rightarrow Y$ et $s\in G_{Nis}F(Y)$.
$\\$$\\$

(ii) Le morphisme structural $\eta_{F}$ est donn\'{e} par:\\
$[\eta_{F}(s)]_{x}=F\{{l}_{x}^{h}\}(s)$ par un \'{e}lement s de F(X).
$\\$
$\\$

(iii) Le morphisme structural $\mu_{F}$ coïncide avec la projection sur les composantes associ\'{e}s aux points ferm\'{e}s des $X_{x}^{h}$ via le morphisme naturel:\\
$(\mathcal{G}_{Nis}G_{Nis}F)(X)=\prod_{x\in X} \prod_{z \in X_{x}^{h}} F\{(X_{x}^{h})_{z}^{h}\}$.
$\\$
$\\$
La r\'{e}solution cosimpliciale de Godement est par d\'{e}finition le faisceau Nisnevich simplicial:\\
$\mathcal{G}^{\bullet}F=B^{\bullet}(\mathcal{G}_{Nis},F)$.
On note ensuite $G_{Nis}^{\bullet}F$ le complexe qui s'en d\'{e}duit dont les diff\'{e}rentielles sont les sommes altern\'{e}es des cofaces.

\begin{proposition}
La famille form\'{e}e des foncteurs fibres $F\rightarrow F\{X_{x}^{h}\}$ est conservative et le morphisme d'augmentation $F\rightarrow G_{Nis}^{\bullet}F$ est un quasi-isomorphisme de complexes de faisceaux Nisnevich.
\end{proposition}

\begin{proof}
cf [20] Partie III.
\end{proof}

\begin{proposition}([15bis] Prop 3.8)
Il existe une monade canonique $\mathcal{G}_{Nis}^{tr}$ de la cat\'{e}gorie $Sh_{Nis}^{tr}(\mathcal{S})$ rendant le diagramme suivant commutatif:\\
$\xymatrix{Sh_{Nis}^{tr}(\mathcal{S})\ar[d]^{\mathcal{G}_{Nis}^{tr}}\ar[r]&Sh_{Nis}(\mathcal{S})\ar[d]^{\mathcal{G}_{Nis}}\\Sh_{Nis}^{tr}(\mathcal{S})\ar[r]&Sh_{Nis}(\mathcal{S})}$

\end{proposition}

Cette proposition assure que le complexe $G_{Nis}^{\bullet}F$ est canoniquement muni de transferts. Il ne nous reste plus qu'\`{a} voir que la r\'{e}solution de Godement est compatible \`{a} la structure tensorielle.

\begin{definition}
On appelle faisceau Nisnevich quasi-monoïdal sym\'{e}trique un faisceau Nisnevich F avec pour tout sch\'{e}ma X,Y un morphisme associatif sym\'{e}trique:\\
$\oslash_{X,Y}^{F}:F(X)\otimes F(Y)\rightarrow F(X\times Y)$ fonctoriel.\\
On note $Sh_{Nis,\otimes}(\mathcal{S})$ la cat\'{e}gorie de tels faisceaux. On a les mêmes d\'{e}finitions pour les faisceaux  Nisnevich avec transferts.

La monade $\mathcal{G}_{Nis}$ induit une monade sur $Sh_{Nis,\otimes}(\mathcal{S})$. 
\end{definition}
$\\$En effet, on a :$\\$
$\oslash_{X,Y}^{\mathcal{G}_{Nis}F}:\mathcal{G}_{Nis}F(X)\otimes \mathcal{G}_{Nis}F(Y)\rightarrow \mathcal{G}_{Nis}F(X\times Y)$\\
donn\'{e}s par les relations $[\oslash_{X,Y}^{F}(s\otimes t)]=F\{\{m \}_{X,Y,e}^{h}\}[\oslash_{X_{x}^{h},Y^{h}_{y}}^{F}(s_{x}\otimes t_{y})]$
 avec $s\in\mathcal{G}_{Nis}F(X)$, $ t\in\mathcal{G}_{Nis}F(Y)$ et e un point du produit de $X\times Y$ de projection x et y.
 
 On a alors la proposition  analogue:
 \begin{proposition}([15bis] Prop 3.11)\\
  $\mathcal{G}_{Nis}^{tr}$ induit une monade sur la cat\'{e}gorie $Sh_{Nis,\otimes}^{tr}(\mathcal{S})$ rendant le diagramme suivant commutatif:\\
$\xymatrix{Sh_{Nis,\otimes}^{tr}(\mathcal{S})\ar[d]^{\mathcal{G}_{Nis}^{tr}}\ar[r]&Sh_{Nis}(\mathcal{S})\ar[d]^{\mathcal{G}_{Nis}}\\Sh_{Nis,\otimes}^{tr}(\mathcal{S})\ar[r]&Sh_{Nis}(\mathcal{S})}$

 \end{proposition}

\section{Construction du foncteur de r\'{e}alisation}
\subsection{Complexe de de Rham logarithmique}$\\$

On rappelle la d\'{e}finition du complexe de de Rham logarithmique suivant [7]. Soit k un corps de caract\'{e}ristique 0. 
 D  est appel\'{e} diviseur \`{a} croisements normaux, si le morphisme d'inclusion est localement isomorphe pour la topologie \'{e}tale \`{a} la r\'{e}union d'hyperplans de coordonn\'{e}es dans $\mathbb{A}^{n}$  . 
 Soit donc D un tel diviseur dans $\bar{X}$ un sch\'{e}ma lisse s\'{e}par\'{e} sur k. 
On pose X=$\bar{X}$-D et j l'inclusion de X dans $\bar{X}$.\\

On d\'{e}signe par $\Omega^{1}_{\bar{X}}(log D)$ le sous $\mathcal{O}-module$ de $j_{*}\Omega^{1}_{X}$ engendr\'{e} par $\Omega^{1}_{\bar{X}}$ et par les $\frac{dz_{i}}{z_{i}}$ pour $z_{i}$ une \'{e}quation locale d'une composante irreductible locale de Y.\\
Par d\'{e}finition, le faisceau des p-formes diff\'{e}rentielles logarithmiques sur $\bar{X}$ le long de D est $\Lambda^{p}\Omega^{1}_{\bar{X}}(Y)$.
On note  $\Omega^{\bullet}_{\bar{X}}(log D)$ le complexe de de Rham logarithmique de $\bar{X}$ le long de D.

Le complexe de de Rham logarithmique est contravariant par rapport au couple (X,$\bar{X}$), i-e \`{a} un diagramme de la forme:
\xymatrix{X\ar[r]\ar[d]&\bar{X}\ar[d]\\Y\ar[r]&\bar{Y}}\\

On d\'{e}signe par $W_{n}(\Omega^{\bullet}_{\bar{X}}(log D))$ le sous module de $\Omega^{\bullet}_{\bar{X}}(log D)$ engendr\'{e} par les combinaisons lin\'{e}aires de formes diff\'{e}rentielles de la forme:\\

$$\alpha\wedge\frac{dz_{i(1)}}{z_{i(1)}}\wedge\frac{dz_{i(2)}}{z_{i(2)}}\wedge .. \wedge\frac{dz_{i(m)}}{z_{i(m)}}  (m\leq{n})$$
 \\
avec $z_{i(j)}$ \'{e}quations locales de composantes locales disctinctes $Y_{j}$ de Y  et $\alpha$ sans pôles. $\\$

Ce complexe est bifiltr\'{e}:
\begin{enumerate}
	\item

par la filtration de Hodge F telle que $F^{p}( \Omega^{\bullet}_{\bar{X}}(log D))=\sigma_{\geq{p}}\Omega^{\bullet}_{\bar{X}}(log D)$ o\`{u} $\sigma_{\geq{p}}$ est le tronqu\'{e} b\^{e}te d\'{e}fini par:
$\sigma_{\geq{p}}(K^{\bullet})^{n}=
\begin{cases}$
 0 \text{ si } $n<p$ $\\
   K^{n} \text{ si } n\geq p
\end{cases}$.
\item 
par la filtration croissante par le poids W, donn\'{e}e par les sous-complexes $W_{n}(\Omega^{\bullet}_{\bar{X}}(log D))$.
\end{enumerate}
\begin{proposition}$\\$
(i) $\mathbb{H}(\bar{X},\Omega^{\bullet}_{\bar{X}}(log D))$ coïncide avec $H(X,\mathbb{C})$.$\\$
(ii) On a un  isomorphisme canonique $H_{DR}(X\times \mathbb{A}^{1}/k)\rightarrow H_{DR}(X/k)$ o\`{u} le morphisme est donn\'{e} par le pullback des formes diff\'{e}rentielles $pr_{1}^{*}: pr_{1}^{-1}\Omega^{\bullet}_{X}\rightarrow  \Omega^{\bullet}_{X\times \mathbb{ A}^{1}}$.\\
(iii) La cohomologie de de Rham v\'{e}rifie Mayer-Vietoris.
\end{proposition} 
\begin{proof}
(i) cf lemme 3.14 [6].$\\$
(iii) r\'{e}sulte des propri\'{e}t\'{e}s d'hypercohomologie.$\\$
(ii) Soit $p=pr_{1}:X\times \mathbb{A}^{1}_{k}\rightarrow X$. Montrons que 
$pr^{*}: \Omega^{\bullet}_{X}\rightarrow  pr_{*}\Omega^{\bullet}_{X\times A^{1}_{k}}$ est un quasi-isomorphisme. En effet, comme l'application p est affine, on peut calculer $H_{DR}^{p}~(X\times \mathbb{A}^{1}/k)$, comme la cohomologie du complexe simple associ\'{e} au complexe double de de Rham-Cech de $\Omega^{\bullet}_{X\times \mathbb{A}^{1}}$ associ\'{e} au recouvrement affine de $p^{-1}(U)$ de $X\times \mathbb{A}^{1}_{k}$ o\`{u} U est un recouvrement affine de X.$\\$ Or ce complexe simple n'est rien d'autre que le complexe  simple associ\'{e} au complexe double de de Rham-Cech de $p_{*}\Omega^{\bullet}_{X\times \mathbb{A}^{1}}$, et donc associ\'{e} au recouvrement affine de U de X, et donc c'est aussi l'hypercohomologie du complexe $p_{*}\Omega^{\bullet}_{X\times \mathbb{A}^{1}}$ sur X. Si on a montr\'{e} que ces complexes sont quasi-isomorphes via $p^{*}$ \`{a} $\Omega^{\bullet}_{X}$, on aura alors que c'est un isomorphisme en cohomologie de de Rham. $\\$
Soit $\alpha \in p_{*}\Omega^{l}_{X\times \mathbb{A}^{1}}$. Alors $\alpha$ s'\'{e}crit sous la forme:
\begin{center}
$\alpha=\alpha'+dt\wedge \beta$
\end{center}
o\`{u}
$\alpha'$ et $\beta$ sont des sections de $p_{*}^{*}\Omega^{l}_{X\times \mathbb{A}^{1}}$, $p_{*}^{*}\Omega^{l-1}_{X\times \mathbb{A}^{1}}$ respectivement.\\ $\beta$ s'\'{e}crit localement:
$\beta=\sum {t^{i}\beta_{i},i<k+1}$ o\`{u} $\beta_{i}$ sont des sections de $\Omega^{l-1}_{X}$. Comme $ \text{car k}=0$, on a $t^{i}=\frac{1}{i+1}dt^{i+1}$, et donc:\\
\begin{center}
$t^{i}dt\wedge \beta_{i}=d(\frac{1}{i+1}t^{i+1}\beta_{i})-\frac{1}{i+1}t^{i+1}d\beta_{i}$.
\end{center}

Il en r\'{e}sulte que toute section $\alpha \in p_{*}\Omega^{l}_{X\times \mathbb{A}^{1}}$ est modulo une forme exacte, dans $p_{*}^{*}\Omega^{l}_{X\times \mathbb{A}^{1}}$.
 Soit maintenant $\alpha=\sum t^{i}\alpha_{i}\in p_{*}^{*}\Omega^{l}_{X}$
 telle que $d\alpha=0$.\\
\begin{center}
$\sum t^{i}d\alpha_{i}+\sum it^{i-1}dt\wedge\alpha_{i}=0$
\end{center}
 Comme $\text{car k}= 0$, on a que $\alpha_{i}=0$ pour $i>0$ et $\alpha \in \Omega^{l}_{X}$.
Ceci montre la surjectivit\'{e} en cohomologie.
L'injectivit\'{e} proc\`{e}de de la m\^{e}me mani\`{e}re.

\end{proof}

Le cas qui nous int\'{e}resse ici est si on prend X sch\'{e}ma lisse s\'{e}par\'{e} sur k. On a alors par Hironaka [1], qu'il existe $\bar{X}$ une compactification projective lisse telle que $\bar{X}$-X est un diviseur \`{a} croisements normaux.

\subsection{Construction de la r\'{e}alisation}$\\$
Suivant Deligne-Goncharov [10], la construction se d\'{e}compose selon plusieurs \'{e}tapes.
On consid\`{e}re $X^{\bullet}$ un complexe born\'{e} d'objets de SmCor(k).\\
 On consid\`{e}re une compactification lisse de chacun des $X^{n}$  et on montre que le complexe $X^{\bullet}$ se prolonge au  complexe $\bar{X}^{\bullet}$. \\
 Ensuite, on prend pour chaque $\bar{X}^{n}$ le complexe  de de Rham logarithmique, et on va consid\'{e}rer un complexe $K^{\bullet}$ qui repr\'{e}sente $R\Gamma(\bar{X}^{n},\Omega^{\bullet}_{\bar{X}}(log D^{n}))$ qui est lui-m\^{e}me  bifiltr\'{e}; pour cela  la r\'{e}solution de Godement fera l'affaire. On en d\'{e}duira un complexe double bifiltr\'{e} qui nous fournira la r\'{e}alisation voulue \`{a} quelques d\'{e}tails pr\`{e}s. Enfin, il s'agit de passer \`{a} la cat\'{e}gorie des motifs mixtes et d'ajouter la structure tensorielle.
 \begin{theorem} On a un foncteur de r\'{e}alisation
 $R_{DR}:\left\{\xymatrix{SmCor(k)^{op}\ar[r]&D^{+}(k_{bifilt})\\X\ar[r]&Dec R\Gamma(\bar{X},\Omega^{\bullet}_{\bar{X}}(log D))}\right\}$
 qui s'\'{e}tend en un foncteur  triangul\'{e} tensoriel sur $DM_{gm}(k)$.
\end{theorem}
Nous utilisons la topologie Nisnevich (\'{e}tale fonctionne aussi), essentiellement parce que les r\'{e}sultats \'{e}tablis pour la monade de Godement marchent pour ces deux topologies et non pour la topologie de Zariski. L'hypercohomologie du complexe de de Rham est la même. En effet, on rappelle le r\'{e}sultat suivant:
\begin{proposition}
 Soit $\mathcal{F}$ un faisceau quasi-coh\'{e}rent sur $X_{zar}$, il induit un faisceau $\mathcal{F}_{\text{Nis}}$,
 on a alors que $ H^{*}(X,F) \rightarrow H^{*}(X_{\text{Nis}},F_{\text{Nis}})$ est un isomorphisme.
\end{proposition}
\begin{proof}
 SGA 4 Exp VII Prop 4.3, o\`{u} c'est d\'{e}montr\'{e} pour la topologie \'{e}tale et donc en particulier pour celle de Nisnevich.
\end{proof}$\\$
Avec cette proposition, l'hypercohomologie du complexe de de Rham s'en d\'{e}duit par la suite spectrale standard.

On commence maintenant avec le lemme suivant qui nous assure que les correspondances finies se prolongent aux compactifications.

\begin{lemma}
Soit $\Gamma: X\rightarrow Y$ dans SmCor(k). Si $\bar{Y}$ est une compactification projective lisse de Y, il existe $ \bar{X}$ une compactification projective lisse  telle que $\Gamma$ se prolonge de mani\`{e}re unique en $\bar{\Gamma}$ de $\bar{X}$ vers $\bar{Y}$.
\end{lemma}

\begin{proof}

On a seulement besoin de traiter le cas o\`{u} X est connexe et $\Gamma$ est un sous-sch\'{e}ma ferm\'{e} de $X\times Y$, int\`{e}gre et fini dominant sur X. On note alors d le degr\'{e} de $\Gamma$ sur X. D'apr\`{e}s la proposition 3, on a alors un morphisme  de sch\'{e}mas $\gamma:X \rightarrow Sym^{d}(Y)$.
Par Hironaka [13], il existe $ \bar{X}$ une compactification projective lisse  telle que $\gamma$ se prolonge en $\bar{\gamma}:\bar{X}\rightarrow Sym^{d}(\bar{Y})$. Pour une telle compactification, on prend l'adh\'{e}rence de $\Gamma$ dans $\bar{X}\times \bar{Y}$.
Celle-ci est bien finie dominante sur $ \bar{X}$ et fournit le prolongement cherch\'{e}. L'unicit\'{e} vient du fait  que $ \bar{X}$ est int\`{e}gre et $Sym^{d}(\bar{Y})$ s\'{e}par\'{e}.

\end{proof}

On veut d\'{e}sormais montrer que le complexe de de Rham est canoniquement muni de transferts, puis nous passerons apr\`{e}s au cas logarithmique.$\\$$\\$
\begin{proposition} 

 Soit $\Gamma: X\rightarrow Y$ dans SmCor(k).
 \begin{enumerate}
 \item On a alors un morphisme de complexes de faisceaux sur le support de $\Gamma$:\\
  $[\Gamma]: pr_{2}^{*}\Omega^{\bullet}_{Y}\rightarrow  pr_{1}^{*}\Omega^{\bullet}_{X}$.
  \item Le morphisme d\'{e}fini ci-dessus est fonctoriel sur SmCor(k).

\end{enumerate}
\end{proposition}

\begin{proof}
Pour d\'{e}montrer cette assertion, on peut se restreindre au cas o\`{u} X et Y sont affines.

Soit B une A-alg\`{e}bre finie. Soient L et K leurs corps de fractions respectifs. L/K est finie s\'{e}parable comme on est en caract\'{e}ristique nulle. On va montrer que l'application naturelle\\ $i:\Omega^{\bullet}_{A/k}\stackrel{A}{\otimes} L\rightarrow \Omega^{\bullet}_{B/k}\stackrel{B}{\otimes} L $ est bijective.

On remarque tout d'abord que la structure  de complexe sur L sur $\Omega^{\bullet}_{B/k}\stackrel{B}{\otimes} L$  est donn\'{e}e par:\\
$d(\frac{\omega}{a})=\frac{ad\omega-da\wedge\omega}{a^{2}}$.
De m\^{e}me, on met une structure de complexe sur K sur $C:=\Omega^{\bullet}_{A/k}\stackrel{A}{\otimes} K$. On va maintenant mettre un structure de complexe sur 
$C'=C\stackrel{K}{\otimes} L$  de telle façon que i devienne un morphisme de complexes et ensuite construire un inverse \`{a} i.
Premi\`{e}rement, on a comme L/K est s\'{e}parable, que la diff\'{e}rentielle $d:K\rightarrow \Omega^{1}_{A/k}\stackrel{A}{\otimes} K$ s'\'{e}tend en une d\'{e}rivation $d^{*}:L\rightarrow \Omega^{1}_{A/k}\stackrel{A}{\otimes} L$. On a donc la structure que l'on voulait sur C' en posant:\\
\begin{center}
$d^{*}(\omega\otimes l)=ld\omega+d^{*}l\wedge\omega$ $(l\in L,\omega \in E)$.
\end{center}
Il reste \`{a} voir que i est un morphisme de complexes. Pour cela, il suffit de voir que le diagramme suivant commute:\\
\begin{center}
$\xymatrix{&\Omega^{1}_{A/k}\stackrel{A}{\otimes} L\ar[dd]^{i}\\L\ar[ur]^{d^{*}}\ar[dr]_{d}\\&\Omega^{1}_{B/k}\stackrel{B}{\otimes} L}$
\end{center}$\\$
Cela vient du fait que  $i\circ d^{*}$ et d sont deux d\'{e}rivations de $L\rightarrow \Omega^{1}_{B/k}\stackrel{B}{\otimes} L$ et coïncident sur K.
La propri\'{e}t\'{e} universelle de $\Omega^{\bullet}_{A/k}$ nous permet de  construire le morphisme inverse $\lambda: \Omega^{1}_{B/k}\stackrel{B}{\otimes} L\rightarrow C'$.

\begin{lemma}
Il existe une unique application $Tr_{B/A}:\Omega^{\bullet}_{B/k}\stackrel{B}{\otimes} L\rightarrow \Omega^{\bullet}_{A/k}\stackrel{A}{\otimes} K$ telle que:
\begin{enumerate}
\item
$Tr_{B/A}$ est la trace usuelle de L vers K en dimension 0.
\item
$Tr_{B/A}$ est $\Omega^{\bullet}_{A/k}$-lin\'{e}aire (i-e additive et $Tr_{B/A}(\omega\wedge\eta)=\omega\wedge Tr_{B/A}(\omega))$.
\end{enumerate}
\end{lemma}
 \begin{proposition}
$Tr_{B/A}$ est un morphisme de complexes.
\end{proposition}
\begin{proof}
Supposons $z\in L$. Soit $f(z)=z^{n}-\sum{b_{i}z^{n-i}}$  son polynôme minimal sur K. Si $n^{*}=[L:K]$, alors $Tr(z)=\frac{n^{*}}{n}b_{1}$. De plus, on a $df(z)=0$, d'o\`{u} $dz=\frac{1}{f'(z)}\sum{z^{n-i}db_{i}}$. Donc par d\'{e}finition de Tr, $Tr(dz)=\sum{Tr(\frac{z^{n-i}}{f'(z)})db_{i}}$, ce qui est exactement $\frac{n^{*}}{n}db_{1}$ par [26,III,§6].Donc Tr(dz)=dTr(z). Maintenant si $\eta\in \Omega^{\bullet}_{B/k}\stackrel{B}{\otimes} L$, $\eta=z\omega$ avec $z\in L$ et $\omega \in \Omega^{\bullet}_{B/k}$ par la proposition, donc $Tr(d\eta)=Tr(dz\wedge\omega+zd\omega)=dTr(z)\wedge\omega+Tr(z)d\omega=d(Tr(z)\omega)=dTr(\eta)$, ce qu'on voulait.
 \end{proof}
 \begin{proposition}
$Tr_{B/A}$ envoie  $\Omega^{\bullet}_{B/k}$ vers $\Omega^{\bullet}_{A/k}$ si A est lisse.
\end{proposition}
 \begin{proof}
 On a d\'{e}j\`{a} que $\Omega^{n}_{A/k}$ est projectif de type fini, comme A lisse. Comme A est normal, int\`{e}gre et L/K s\'{e}parable, la fermeture int\'{e}grale A', de A dans L est une
A-alg\`{e}bre finie. De plus, si on prend un ouvert U, dont le compl\'{e}mentaire est de codimension au moins 2 dans A, le morphisme de restriction est un isomorphisme. Donc, il suffit de montrer que $Tr_{B/A}$ envoie $\Omega^{\bullet}_{A'/k}$ dans $\Omega^{\bullet}_{A/k}\otimes K$ et s'\'{e}tend aux points de codimension 1 de Spec A. On se ram\`{e}ne donc au cas o\`{u} A est un trait. Comme on peut faire n'importe quelle extension \'{e}tale sur A, on peut supposer A strictement hens\'{e}lien.

Dans ce cas-l\`{a}, i-e A un trait strictement hens\'{e}lien, on a:\\ $A'=A[t]/(t^{n}-z)$ o\`{u} z est un g\'{e}n\'{e}rateur de l'id\'{e}al maximal et n entier positif. Alors $dt=\frac{1}{nt^{n-1}}dz$ (car k=0) et le groupe de Galois de A' sur A, agit par $\sigma(t)=\zeta t$ o\`{u} $\zeta$ est une racine primitive de l'unit\'{e}. Donc $Tr(t^{m})=0$ sauf si n divise m. On en d\'{e}duit $Tr(t^{m}dt)=0$ si $m\neq-1[n]$ ou $z^{(m-n+1)/n}dz$ si $m=-1[n]$.

 \end{proof}
 Etant donn\'{e}e une correspondance $\Gamma\in Cor(X,Y)$, on notera $[\Gamma]$ la fl\`{e}che $Tr_{\Gamma/X}$ compos\'{e}e avec le morphisme canonique 
 $ pr_{2}^{*}(\Omega^{\bullet}_{Y})\rightarrow  pr_{1}^{*}(\Omega^{\bullet}_{\Gamma})$.
 Montrons d\'{e}sormais la compatibilit\'{e} par rapport \`{a} la composition des correspondances.
 Soit donc, $W\in X\times Y$, $ W'\in Y\times Z$,avec X, Y et Z affines lisses. En effet, un th\'{e}or\`{e}me de Voevodsky, nous assure que l'on peut se restreindre \`{a} une telle sous-cat\'{e}gorie ([5 bis]). Comme $\Omega^{n}_{X}$ est localement libre, on peut remplacer X par son point g\'{e}n\'{e}rique Spec(F).
 On a donc le diagramme suivant:
 \xymatrix{W\subseteq X\times Y\ar[d]\\X=Spec(F)}
  $\\$$\\$On peut de même remplacer X par W et on  se ram\`{e}ne donc au cas o\`{u} W est le graphe de l'inclusion $i_{X}:X\rightarrow Y$ avec X, Y  affines et lisses.
 Maintenant, on consid\`{e}re une chaîne de sous-vari\'{e}t\'{e}s lisses de codimension un:
 $X=X_0\subset X_{1}\subset...\subset X_{N}=Y$.$\\$
Montrons par r\'{e}currence descendante sur N, que l'on peut se ramener 
 au cas o\`{u} X est de codimension 1 dans Y.
 Supposons la propri\'{e}te vraie au rang N-j et montrons la au rang N-j-1.
 On a alors :
\begin{center}
 $X_{N-j-1}\stackrel{i_{N-j-1}}{\rightarrow} X_{N-j}\stackrel{i_{N-j}}{\rightarrow} Y$.
\end{center}
 Par hypoth\`{e}se de r\'{e}currence, on a $[W']\circ i_{N-j}=[W'\circ i_{N-j}]$ et $i_{N-j-1}$ est une immersion de codimension 1 entre sous-sch\'{e}mas lisses, donc on peut appliquer l'hypoth\`{e}se de r\'{e}currence.
Maintenant, quitte \`{a} remplacer Z par la normalisation de W', on peut supposer que W' est la transpos\'{e}e du graphe d'un F-morphisme fini surjectif de Z vers Y. Comme car k=0 et Y et Z lisses, un tel morphisme est \'{e}tale, on peut donc  prendre Y=Spec(F[x])  et Z=Spec(F[t]) avec $t^{n}=x$.$\\$$\\$
 Alors la composition est $i_{X}\circ W'$ est nx ,  et on a juste \`{a} v\'{e}rifier en degr\'{e} z\'{e}ro. Dans ce cas, c'est juste le morphisme trace sur le faisceau structural et en degr\'{e} un, o\`{u} c'est le calcul fait dans la proposition pr\'{e}c\'{e}dente.

\end{proof} $\\$

\textbf{Remarque}:\\On vient donc de voir que le complexe de de Rham est canoniquement muni de transferts. On a même vu dans la preuve que l'on a des transferts sur les formes diff\'{e}rentielles m\'{e}romorphes. L'extension au cas logarithmique se fait donc de la facon suivante:$\\$
-Soit $j_{X}$ (resp $j_{Y}$) l'inclusion de X dans $\bar{X}$ (et pareillement pour Y), une compactification de Y étant fix\'{e}e.$\\$
-D'apres le lemme 6, on a vu que si on a $\Gamma: X\rightarrow Y$ dans SmCor(k) et est irr\'{e}ductible, on a un unique $\bar{\Gamma}: \bar{X}\rightarrow \bar{Y}$ dans SmCor(k) qui prolonge $\Gamma$ une fois que l'on a fix\'{e} une compactifications $\bar{Y}$ de Y.$\\$
-On a alors d\'{e}j\`{a} un morphisme pour $\Gamma$: \\ $[\Gamma]: j_{Y,*}\Omega^{\bullet}_{Y}\rightarrow  j_{X,*}\Omega^{\bullet}_{X}$, et donc en particulier il suffit de voir que  $[\bar{\Gamma}]$ envoie $\Omega^{\bullet}_{\bar{Y}}(log)$ sur $\Omega^{\bullet}_{\bar{X}}(log)$. Comme \`{a} nouveau, pour tout $i\geq0$, $\Omega^{i}_{\bar{X}}(log)$ est localement libre, on se ram\`{e}ne \`{a} nouveau au cas o\`{u} $\bar{X}\rightarrow \bar{Y}$ est un morphisme fini de courbes au-dessus d'une certaine extension de corps F, avec le diviseur \`{a} l'infini t=0 et z=0 et $t^{n}=z$ avec les mêmes notations que ci-dessus et on a juste besoin de faire le calcul direct pour $\frac{dt}{t}=\frac{dz}{nz}$ ce qu'on voulait.

$\\$
Il ne nous reste donc plus qu'\`{a} montrer que ça ne d\'{e}pend pas de la compactification, on a alors la proposition suivante:

\begin{proposition}
Soit $\Gamma: X\rightarrow Y$ dans SmCor(k) et $\bar{X}_{1}$, $\bar{X}_{2}$, $\bar{Y}_{1}$ et $\bar{Y}_{2}$ des compactifications de X et Y.
On a alors que $\bar{\Gamma}: \bar{X_{1}}\rightarrow \bar{Y}_{1}$ est ind\'{e}pendant des compactifications choisies.
\end{proposition}

\begin{proof}
On pose $\bar{Z}:=\bar{X}_{1}\times\bar{X}_{2}$, puis on note $\bar{X}_{3}$ une d\'{e}singularisation de l'adh\'{e}rence de X dans $\bar{Z}$, on a alors le diagramme suivant:\\
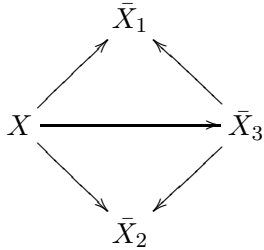
\begin{center}
\xymatrix{&\bar{X}_{1}\\ X\ar[dr]\ar[rr]\ar[ur]&&\bar{X}_{3}\ar[ul]\ar[dl]\\&\bar{X}_{2}}
\end{center}
$\\$
  ce qui nous permet de nous ramener au cas o\`{u} on a un diagramme du type:\\
  
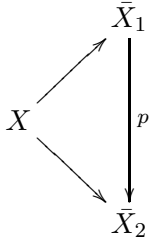
\xymatrix{&\bar{X}_{1}\ar[dd]^{p}\\X\ar[ur]\ar[dr]\\&\bar{X}_{2}} \\
  o\`{u} p est un morphisme propre et pareillement pour Y.
Maintenant,la proposition vient du fait qu'un morphisme propre envoie les formes \`{a} pôles logarithmiques sur les formes \`{a} pôles logarithmiques et  qu'une  forme rationnelle sur $\bar{ X}$ qui s'envoie sur une forme \`{a} pôles logarithmiques sur  
$\bar {X'}$ a elle-même des pôles logarithmiques  sur $\bar X$.

\end{proof}
On a alors d'apr\`{e}s la proposition 17 que
le morphisme $[\bar{\Gamma}]: pr_{2}^{*}(\Omega^{\bullet}_{\bar{Y}}(log))\rightarrow  pr_{1}^{*}(\Omega^{\bullet}_{\bar{X}}(log))$ s'\'{e}tend aux   r\'{e}solutions flasques canoniques
\`{a} $\mathcal{G}_{\text{\'{e}t}}^{\bullet}(\Omega^{\bullet}(log))$.
On repr\'{e}sente alors $ R \Gamma( \bar{X},\Omega^{\bullet}_{\bar{X}}(\log))$ par $\pi_{\bar{X}*}\mathcal{G}_{\text{\'{e}t}}^{\bullet}\Omega^{\bullet}_{\bar{X}}(log)$. Pour assurer l'ind\'{e}pendance par rapport aux choix des compactifications, on prend la limite inductive sur toutes les compactifications, qui est un syst\`{e}me essentiellement constant par ci-dessus. On a donc besoin, du lemme suivant:
\begin{lemma}
La cat\'{e}gorie  I compos\'{e}e des couples $(X,\bar{X})$ avec X lisse sur k et $\bar{X}$ compactification lisse  et les morphismes naturels entre couples, est cofiltrante.
\end{lemma}

\begin{proof}
On a d\'{e}j\`{a} vu dans la proposition pr\'{e}c\'{e}dente que l'on a pour i, j, k des objets de I:

$
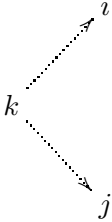
\xymatrix{&i\\k\ar@{.>}[dr]\ar@{.>}[ur]\\&j}$

Il ne nous reste \`{a} montrer que l'on a un diagramme du type:

$
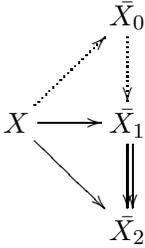
\xymatrix {&\bar{X_{0}}\ar@{.>}[d]\\X\ar@{.>}[ur]\ar[r]\ar[dr]&\bar{X_{1}}\ar[d]\ar@<2pt>[d] \\&\bar{X_{2}}}$

Dans ce cas, on choisit pour $\bar{X_{0}}$ une d\'{e}singularisation de  l'adh\'{e}rence de X dans $\bar{X_{1}}\times_{\bar{X_{2}}}\bar{X_{1}}$.

\end{proof}
On a  donc un foncteur $\tilde{R}_{DR}:C^{b}(SmCor/k)^{op}\rightarrow C^{b}(C^{+}(k_{bifilt}))$ ainsi que des foncteurs:\\ 

 $Tot:C^{b}(C^{+}(k_{bifilt}))\rightarrow C^{+}(k_{bifilt})$ et $Dec_{W}:C^{+}(k_{bifilt})\rightarrow C^{+}(k_{bifilt})$.

On note alors $R_{DR}$ le foncteur composé que l'on prolonge naturellement en un foncteur triangul\'{e} de:\\

\begin{center}
$R_{DR}:K^{b}(SmCor/k)^{op}\rightarrow D^{+}(k_{bifilt})$
\end{center}

En utilisant l'invariance par homotopie  et la suite de Mayer-Vietoris pour la cohomologie de de Rham d'une part et d'autre part le fait que  $D^{+}(k_{bifilt})$ est pseudo-ab\'{e}lienne d'apres [3]. On en d\'{e}duit que le foncteur s'\'{e}tend en un foncteur triangul\'{e} $R_{DR}:DM^{eff}_{gm}(k)^{op}\rightarrow D^{+}(k_{bifilt})$.

Il ne reste plus qu'\`{a} v\'{e}rifier la structure tensorielle:

\begin{lemma}
Le foncteur triangul\'{e} $R_{DR}$ est tensoriel, on a donc pour tout X, Y k-sch\'{e}mas un isomorphisme:$\\$
\begin{center}
$R_{DR}(X)\otimes R_{DR}(Y)\rightarrow R_{DR}(X\times Y).$
\end{center}
 $\\$

\end{lemma}

\begin{proof}
On va d'abord montrer que le foncteur est quasi-tensoriel;
 on a d\'{e}j\`{a}  que $\Omega^{\bullet}_{\bar{X}}(log)$  est muni de transferts et est quasi-monoïdal sym\'{e}trique, donc par la proposition 18, on obtient que $\mathcal{G}_{\text{\'{e}t}}^{\bullet}\Omega^{\bullet}_{\bar{X}}(log)$ est \'{e}galement quasi-monoïdal sym\'{e}trique. Cela entraîne que le foncteur
  $\tilde{R}_{DR}:SmCor/k\rightarrow C^{+}(k_{bifilt})$ \\est quasi-monoïdal sym\'{e}trique, d'o\`{u} un morphisme canonique de foncteurs sur $SmCor/k\otimes SmCor/k$:\\
 
\begin{center}
 $\otimes:\tilde{R}_{DR}(-)\otimes\tilde{R}_{DR}(-)\rightarrow
  \tilde{R}_{DR}(-\times -)$ 
\end{center}
$\\$
  associatif et commutatif. Ce dernier nous fournit des morphismes de foncteurs associatifs et commutatifs:\\
  
\begin{center}
\xymatrix{\mathcal{C}\tilde{R}_{DR}(-)\otimes \mathcal{C}\tilde{R}_{DR}(-)\ar[d]&\mathcal{C}[\tilde{R}_{DR}(-)\otimes\tilde{R}_{DR}(-)]\ar[l]_{\otimes^{EML}}\ar[r]^{\mathcal{C}\otimes}&\mathcal{C}\tilde{R}_{DR}(-\times-)\ar[d]\\\tilde{R}_{DR}(-)\otimes\tilde{R}_{DR}(-)&&\tilde{R}_{DR}(-\times -)}
\end{center}$\\$
  o\`{u} $\otimes^{EML}$ d\'{e}signe la transformation d'Eilenberg-MacLane [10 1/2]. On a donc des morphismes de bifoncteurs associatifs et commutatifs:\\

\begin{center}
\xymatrix{C^{b}[\mathcal{C}\tilde{R}_{DR}(-)\otimes \mathcal{C}\tilde{R}_{DR}(-)]&C^{b}[\mathcal{C}[\tilde{R}_{DR}(-)\otimes\tilde{R}_{DR}(-)]]\ar[l]_{C^{b}\otimes^{EML}}\ar[r]^{C^{b}\mathcal{C}\otimes}&C^{b}[\mathcal{C}\tilde{R}_{DR}(-\times-)]\ar[d]\\\textbf{R}_{DR}(-)\otimes\textbf{R}_{DR}(-)\ar[u]&&\textbf{R}_{DR}(-\times -)}
\end{center}
  Comme en plus, le morphisme d'augmentation:\\
 $\Omega^{\bullet}_{k}(log)\rightarrow \mathcal{G}_{\text{\'{e}t}}^{\bullet}\Omega^{\bullet}_{k}(log)$  est un quasi-isomorphisme qui rend le diagramme commutatif suivant:\\
\begin{center}
   \xymatrix{\textbf{R}_{DR}(k)\otimes \textbf{R}_{DR}(-)&C^{b}[\mathcal{C}[\tilde{R}_{DR}(k)\otimes\tilde{R}_{DR}(-)]]\ar[l]_{C^{b}\otimes^{EML}}\ar[r]^{C^{b}\mathcal{C}\otimes}&\textbf{R}_{DR}(k\times-)\ar[d]\\\Omega^{\bullet}_{k}(log)\otimes\textbf{R}_{DR}(-)\ar[u]\ar[rr]&&\textbf{R}_{DR}(-)}
\end{center}

Le fait que le foncteur est tensoriel provient alors  du fait que le morphisme de Künneth en cohomologie de de Rham est un isomorphisme pour les vari\'{e}t\'{e}s lisses sur un corps k, ce qui entraîne l'isomorphisme pour tout motif mixte.

\end{proof}
En regardant le motif  r\'{e}duit de $\mathbb{P}^{1}$, on voit qu'il a pour image dans $D^{+}(k_{bifilt})$, k en degr\'{e} de Hodge 1. On a donc obtenu le r\'{e}sultat suivant.
\begin{lemma}
Il existe un isomorphisme $\phi:R_{DR}(\mathbb{Z}(1))\rightarrow (k,1)$.
\end{lemma}

Comme (k,1) est inversible dans $D^{+}(k_{bifilt})$, $R_{DR}$ s'\'{e}tend \`{a} $DM_{gm}(k)$, la cat\'{e}gorie des motifs mixtes g\'{e}om\'{e}triques, en un foncteur triangul\'{e} tensoriel.
$\\$$\\$$\\$$\\$$\\$$\\$$\\$
\textbf{Conclusion}: Nous avons donc construit le foncteur de r\'{e}alisation de de Rham avec toutes les propri\'{e}t\'{e}s voulues. Le point technique fondamental \'{e}tait de mettre explicitement des transferts sur le complexe de de Rham et une fois bien comprise cette machine, l'extension \`{a} $DM_{gm}(k)$ est essentiellement formelle. On notera \'{e}galement que cette approche directe \'{e}vite les soucis dûs aux structures enrichies sur la cohomologie, que l'on rencontre dans les approches de Lecomte-Wach et Cisisnki-D\'{e}glise.

Il ne reste plus qu'une remarque \`{a} faire. Le cas de la r\'{e}alisation rigide n'a pas \'{e}t\'{e} abord\'{e}e ici, mais une partie de l'approche \'{e}nonc\'{e}e ici peut être r\'{e}utilis\'{e}e. Pour obtenir les transferts pour la cohomologie rigide, il suffit de construire un morphisme de transfert pour une correspondance finie irr\'{e}ductible finie W au dessus de X affine lisse.
Dans le cas o\`{u} X est affine lisse, la cohomologie rigide se calcule comme la cohomologie de de Rham d'un 'compl\'{e}t\'{e} faible' d'un rel\`{e}vement de X.$\\$ De plus, on a par [22] que les morphismes finis surjectifs entre schémas affines lisses se relèvent en des morphismes finis surjectifs au niveau des complétés faibles.
Le probl\`{e}me vient alors que la correspondance finie W ne se rel\`{e}ve pas, n'\'{e}tant pas lisse. En revanche,  cela devient possible en utilisant de Jong.

Dans ce cas, on a un diagramme de la sorte:

$\xymatrix{\tilde{W}\ar[r]^{f}&W\ar[d]^{g}\\&X}$$\\$$\\$ avec f propre surjective et g\'{e}n\'{e}riquement \'{e}tale et $\tilde{W}$ est lisse. Dans ce cas, $g\circ f$ est  propre, surjectif et g\'{e}n\'{e}riquement fini entre sch\'{e}mas lisses. On peut donc appliquer ce qu'on a dit ci-dessus, il faut alors voir que c'est ind\'{e}pendant du choix de l'alt\'{e}ration, et que le transfert s'\'{e}tend.
Une fois les transferts mis, le reste  est essentiellement formel.
$\\$$\\$$\\$$\\$$\\$$\\$$\\$$\\$$\\$$\\$$\\$$\\$$\\$$\\$$\\$$\\$$\\$$\\$$\\$$\\$$\\$$\\$$\\$$\\$$\\$$\\$$\\$$\\$$\\$
\textbf{Bibliographie}:$\\$
[1] Y. Andr\'{e}, \textsl{Une introduction aux motifs (motifs purs,motifs mixtes et p\'{e}riodes)}, Panoramas et Synth\`{e}ses, SMF, 2004.$\\$$\\$
[2] M. Artin, \textsl{Grothendieck Topologies}, Seminar Notes. Harvard Univ. Dept.,Spring 1962.$\\$$\\$
[3] P. Balmer, M. Schlichting, \textsl{Idempotent completion of triangulated categories},
J. Algebra 236 (2001), no. 2, p. 819-834.$\\$$\\$
[4] A. Beilinson, V. Vologodsky, \textsl{A DG-guide to Voevodsky's motives}, preprint Arxiv ~(2008).$\\$$\\$
[5] D.C. Cisinski, F. D\'{e}glise, \textsl{Mixed Weil cohomologies}, Preprint (2007). $\\$$\\$
[5 bis] D.C.Cisinski, F.D\'{e}glise, Local and stable homological algebra in Grothendieck  
categories. Homotopy, homolgy, appplications. Volume 11, No. 1, pp. 219-260 (2009).$\\$$\\$
[6] P. Deligne,\textsl{ Equations diff\'{e}rentielles \`{a} points singuliers r\'{e}guliers}. Lecture Notes in Math. $\textbf{163}$
(Springer-Verlag 1970).$\\$$\\$
[7] P. Deligne,\textsl{ Th\'{e}orie de Hodge II}. Publ. Math. IHES 40
(1972), 5-57.$\\$$\\$
[8] P. Deligne, \textsl{Cat\'{e}gories tannakiennes}. in Grothendieck Festschrift vol II. Progress in Math. $\textbf{87}$
Birkhäuser Boston (1990) pp. 111-195.$\\$$\\$
[9] P. Deligne, J.S. Milne, A. Ogus, K.Y. Shih. \textsl{Hodge cycles, motives and Shimura varieties}.
Lecture Notes in Math. $\textbf{900}$ (Springer-Verlag 1982).$\\$$\\$
[10] P. Deligne, A. Goncharov, \textsl{Groupes fondamentaux motiviques de Tate mixtes},~ Ann Sci Ens $\textbf{38}$ 1, (2005) p.1-56.$\\$$\\$
[10 1/2] S. Eilenberg, J.A. Zilber, On products of complexes, Amer. J. Math 75
(1953), p. 200-204.$\\$$\\$
[11] E. M. Friedlander and V. Voevodsky, Bivariant Cycle Cohomology, in [34],
pp. 138-187.$\\$$\\$
[12] W. Fulton,\textsl{ Intersection theory}, Ergebnisse der Mathematik und ihrer Grenzgebiete.
3. Folge. A Series of Modern Surveys in Mathematics, vol. 2, Springer-
Verlag, Berlin, 1998.$\\$$\\$
[13] H. Hironaka,  \textsl{Resolution of singularities of an algebraic variety over a field of characteristic
zero. I, II }, Ann. of Math. (1) \textbf{79} (1964), 109-203 ; ibid. (2) \textbf{79} (1964), p. 205-326.$\\$$\\$
[14] A. Huber, \textsl{Mixed motives and their realization in derived categories}, Lecture Notes in
Mathematics, vol. \textbf{1604}, Springer-Verlag, Berlin, 1995.$\\$$\\$
[15] F. Ivorra,\textsl{ R\'{e}alisation l-adique des motifs mixtes}, Th\`{e}se de doctorat de
l'Universit\'{e} Paris 6 (2005).$\\$$\\$
[15 bis]  F. Ivorra,\textsl{ R\'{e}alisation l-adique des motifs triangul\'{e}s g\'{e}om\'{e}triques} I, Doc Math (2007). $\\$$\\$
[16] U. Janssen, \textsl{Motives, numerical equivalence, and semi-simplicity}. Invent. Math. \textbf{107}: 447-452, 1992.$\\$$\\$$\\$
[17] M. Kashiwara, P. Schapira, \textsl{Categories and sheaves}, Grundlehren der
Math. Wiss. \textbf{332} Springer-Verlag, (2005).$\\$$\\$
[18] H. Kawanoue, \textsl{Toward resolution of singularities over a field of positive characteristic Part I. Foundation of the program: the language of the idealistic filtration}, Arxiv. $\\$$\\$
[19] H. Kawanoue, K. Matsuki,\textsl{ Toward resolution of singularities over a field of positive characteristic (The Idealistic Filtration Program) Part II. Basic invariants associated to the idealistic filtration and their properties }, Arxiv.$\\$$\\$
[19bis] F. Lecomte, N. Wach, Le complexe motivique de de Rham,  Arxiv.$\\$$\\$
[20] M. Levine,\textsl{ Mixed motives}, American Mathematical Society, Providence, RI, 1998.$\\$$\\$
[20bis] D. Lieberman, Numerical equivalence and homological equivalence of algebraic cycles on Hodge manifolds, Amer.J.Math.\textbf{90} (1968), p53-66$\\$$\\$
[21] C. Mazza, V. Voevodsky, C. Weibel, \textsl{Lecture notes on motivic cohomology},
Clay Mathematics Monographs, vol. \textbf{2}, American Mathematical Society
(2006).$\\$$\\$
[22] P. Monsky and G. Washnitzer, Formal Cohomology I
The Annals of Mathematics, Second Series, Vol. \textbf{88}, No. 2 (Sep., 1968), pp. 181-217.$\\$$\\$
[23] Ye. A. Nisnevich, The completely decomposed topology on schemes and associated descent
spectral sequences in algebraic K-theory, Algebraic K-theory: connections with
geometry and topology (Lake Louise, AB, 1987), Kluwer Acad. Publ., Dordrecht, 1989,
pp. 241-342.$\\$$\\$
[24] N. Saavedra Rivano, Cat\'{e}gories tannakiennes, Lecture Notes in Mathematics
\textbf{265}, Springer-Verlag, Berlin, 1972.$\\$$\\$
[25] J-P. Serre  Algèbre locale, Multiplicités, Cours au Collège de France, 1957-1958, redigé
par Pierre Gabriel. Seconde édition, 1965. Lecture Notes in Mathematics, vol. \textbf{11}, Springer-
Verlag, Berlin, 1965.$\\$$\\$
[26] J-P. Serre. Corps locaux. Publications de l'Institut de
Math\'{e}matique de l'Universit\'{e} de Nancago, VIII. Hermann, Paris, 1962.$\\$$\\$
[27] J-P. Serre, Propri\'{e}t\'{e}s conjecturales des groupes de Galois motiviques et des repr\'{e}sentations l-adiques, Motives (Seattle, WA, 1991), Proc. Sympos. Pure Math.,
vol. \textbf{55}, Amer. Math. Soc., Providence, RI, 1994, Part 1. $\\$$\\$
[28] Th\'{e}orie des topos et cohomologie \'{e}tale des schemas I, II, III, Springer-Verlag, Berlin, 1972-
1973, S\'{e}minaire de G\'{e}om\'{e}trie Alg\`{e}brique du Bois-Marie 1963-1964 (SGA 4), Dirig\'{e}
par M. Artin, A. Grothendieck, et J.-L. Verdier. Avec la collaboration de N. Bourbaki,
P. Deligne et B. Saint-Donat, Lecture Notes in Mathematics, Vol. \textbf{269}, \textbf{270}, \textbf{305}.$\\$$\\$
[29] J-L. Verdier, Des cat\'{e}gories d\'{e}riv\'{e}es des cat\'{e}gories ab\'{e}liennes,
Asterisque-Soc. Math. France \textbf{239} (1996).$\\$$\\$
[30] J.L. Verdier, Cat\'{e}gories triangul\'{e}es, \'{e}tat 0. Cohomologie \'{e}tale. S\'{e}minaire de G\'{e}om\'{e}trie
algébrique du Bois-Marie SGA 4 1/2.$\\$$\\$
[31] P. Deligne, SGA 4 1/2, avec la collaboration de J. F. Boutot, A.
Grothendieck, L. Illusie et J-L. Verdier. pp. 262-311. Lecture Notes in Mathematics, \textbf{569}.
Springer-Verlag, Berlin-New York, 1977.$\\$$\\$
[32] V. Voevodsky, Triangulated categories of motives over a field, in Cycles,
transfers, and motivic homology theories, Annals of Mathematics Studies,
vol. \textbf{143}, Princeton University Press, Princeton, NJ, 2000.$\\$$\\$
[33] V. Voevodsky, Homology of schemes, Selecta Math. (N.S.) 2 (1996), no. \textbf{1},
p. 111-153.$\\$$\\$
[34] V. Voevodsky, A. Suslin and E. M. Friedlander, Cycles, transfers, and motivic homology
theories, Annals of Mathematics Studies, vol. \textbf{143}, Princeton University Press, 2000.

\end{document}